\documentclass[11pt]{amsart}
\usepackage{amsmath}
\usepackage{amssymb}
\usepackage{amscd}
\usepackage[mathscr]{eucal}
\def\N{{\mathbb{N}}}
\def\Z{{\mathbb{Z}}}
\def\R{{\mathbb{R}}}
\def\S{{\mathbb{S}}}
\def\Image{\operatorname{Im}}

\def\Cl{\operatorname{Cl}}
\def\diam{\operatorname{diam}}

\def\Isom{\operatorname{Isom}}
\newtheorem{Theorem}{Theorem}[section]

\newtheorem{Lemma}[Theorem]{Lemma}
\newtheorem{Proposition}[Theorem]{Proposition}

\theoremstyle{definition}
\newtheorem{Definition}[Theorem]{Definition}
\newtheorem{Example}[Theorem]{Example}
\newtheorem*{Problem}{Problem}
\newtheorem*{Conjecture}{Conjecture}
\theoremstyle{remark}

\begin{document}
\sloppy
\title{On equivariant homeomorphisms of \\
boundaries of CAT(0) groups \\
and Coxeter groups}
\author{Tetsuya Hosaka}
\address{Department of Mathematics, 
Shizuoka University, Suruga-ku, Shizuoka 422-8529, Japan}
\date{March 30, 2014}
\email{sthosak@ipc.shizuoka.ac.jp}
\thanks{
Partly supported by the Grant-in-Aid for Young Scientists (B), 
The Ministry of Education, Culture, Sports, Science and Technology, Japan.
(No.\ 25800039).}
\begin{abstract}
In this paper, 
we investigate an equivariant homeomorphism of 
the boundaries $\partial X$ and $\partial Y$ of 
two proper CAT(0) spaces $X$ and $Y$ on which a CAT(0) group $G$ acts geometrically.
We provide a sufficient condition and an equivalent condition 
to obtain a $G$-equivariant homeomorphism of 
the boundaries $\partial X$ and $\partial Y$ 
as a continuous extension of the quasi-isometry 
$\phi:Gx_0\rightarrow Gy_0$ defined by $\phi(gx_0)=gy_0$, 
where $x_0\in X$ and $y_0\in Y$.
In this paper, we say that a CAT(0) group $G$ is {\it equivariant (boundary) rigid}, 
if $G$ determines its ideal boundary 
by the equivariant homeomorphisms as above.
As an application, we introduce some examples of (non-)equivariant rigid CAT(0) groups 
and we show that 
if Coxeter groups $W_1$ and $W_2$ are equivariant rigid as reflection groups, 
then so is $W_1 * W_2$.
We also provide a conjecture on non-rigidity of boundaries of some CAT(0) groups.
\end{abstract}
\maketitle
%
\section{Introduction}

In this paper, 
we investigate an equivariant homeomorphism of 
the boundaries of two proper CAT(0) spaces on which a CAT(0) group acts geometrically 
as a continuous extension of a quasi-isometry of the two CAT(0) spaces.

Definitions and details of CAT(0) spaces and their boundaries 
are found in \cite{BH} and \cite{GH}.
A {\it geometric} action on a CAT(0) space 
is an action by isometries which is proper (\cite[p.131]{BH}) and cocompact.
We note that every CAT(0) space on which some group acts 
geometrically is a proper space (\cite[p.132]{BH}).
A group $G$ is called a {\it CAT(0) group}, 
if $G$ acts geometrically on some CAT(0) space $X$.

It is well-known that 
if a Gromov hyperbolic group $G$ acts 
geometrically on a negatively curved space $X$, then 
the natural map $G\rightarrow X$ $(g\mapsto gx_0)$ extends continuously to 
an equivariant homeomorphism of the boundaries of $G$ and $X$.
Also if a Gromov hyperbolic group $G$ acts 
geometrically on negatively curved spaces $X$ and $Y$, then 
the boundaries of $X$ and $Y$ are $G$-equivariant homeomorphic.
Indeed 
the natural map $Gx_0 \rightarrow Gy_0$ $(gx_0\mapsto gy_0)$ extends continuously to 
a $G$-equivariant homeomorphism of the boundaries of $X$ and $Y$.
The boundaries of Gromov hyperbolic groups are quasi-isometric invariant 
(cf.\ \cite{BH}, \cite{CP}, \cite{GH}, \cite{Gr}, \cite{Gr0}).

Here in \cite{Gr0}, Gromov asked whether 
the boundaries of two CAT(0) spaces $X$ and $Y$ are $G$-equivariant homeomorphic 
whenever a CAT(0) group $G$ acts geometrically on the two CAT(0) spaces $X$ and $Y$.
In \cite{BR}, 
P.~L.~Bowers and K.~Ruane have constructed an example that 
the natural quasi-isometry $Gx_0 \rightarrow Gy_0$ $(gx_0\mapsto gy_0)$ 
does not extend continuously to 
any map between the boundaries $\partial X$ and $\partial Y$ of $X$ and $Y$.
Also S.~Yamagata \cite{Y} has constructed a similar example 
using a right-angled Coxeter group and its Davis complex.
Moreover, there is a research by C.~Croke and B.~Kleiner \cite{CK0} on 
an equivariant homeomorphism of the boundaries $\partial X$ and $\partial Y$.

Also, C.~Croke and B.~Kleiner \cite{CK} have constructed a CAT(0) group $G$ 
which acts geometrically on two CAT(0) spaces $X$ and $Y$ 
whose boundaries are not homeomorphic, 
and J.~Wilson \cite{W} has proved 
that this CAT(0) group has uncountably many boundaries.
Recently, C.~Mooney \cite{Moo1} has showed that 
the knot group $G$ of any connected sum of two non-trivial torus knots 
has uncountably many CAT(0) boundaries.

Also, it has been observed by M.~Bestvina \cite{Bes} 
that all the boundaries of a given CAT(0) group are shape equivalent 
and R.~Geoghegan and P.~Ontaneda have proved this in \cite{GO}.
Bestvina has asked the question whether 
all the boundaries of a given CAT(0) group are cell-like equivalent.
This question is an open problem and there are some resent research (cf.\ \cite{Moo2}).

The purpose of this paper is 
to provide a sufficient condition and an equivalent condition 
to obtain a $G$-equivariant homeomorphism between 
the two boundaries $\partial X$ and $\partial Y$ of 
two CAT(0) spaces $X$ and $Y$ on which a CAT(0) group $G$ acts geometrically 
as a continuous extension of the natural quasi-isometry 
$Gx_0\rightarrow Gy_0$ $(gx_0\mapsto gy_0)$, 
where $x_0\in X$ and $y_0\in Y$.

Now we recall the example of Bowers and Ruane in \cite{BR}.
Let $G= F_2\times \Z$ and $X=Y=T\times \R$, 
where $F_2$ is the rank 2 free group generated by $\{a,b\}$ and 
$T$ is the Cayley graph of $F_2$ with respect to the generating set $\{a,b\}$.
Then we define the action ``$\cdot$'' of the group $G$ on the CAT(0) space $X$ by 
\begin{align*}
&(a,0)\cdot(t,r)=(a\cdot t,r), \\
&(b,0)\cdot(t,r)=(b\cdot t,r), \\
&(1,1)\cdot(t,r)=(t,r+1),
\end{align*}
for each $(t,r)\in T\times \R=X$, 
and also define the action ``$*$'' of the group $G$ on the CAT(0) space $Y$ by 
\begin{align*}
&(a,0)*(t,r)=(a\cdot t,r), \\
&(b,0)*(t,r)=(b\cdot t,r+2), \\ 
&(1,1)*(t,r)=(t,r+1),
\end{align*}
for each $(t,r)\in T\times \R=Y$.
Then the group $G$ acts geometrically on the two CAT(0) spaces $X$ and $Y$, 
and the quasi-isometry $g\cdot x_0 \mapsto g*y_0$ 
(where $x_0=(1,0)\in X$ and $y_0=(1,0)\in Y$) 
does not extend continuously to 
any map from $\partial X$ to $\partial Y$.
Indeed for $g_i=a^ib^i\in F_2$, 
$\{g_i^\infty\,|\,i\in\N\}\rightarrow a^\infty$ as $i\rightarrow\infty$ in $\partial T$,
\begin{align*}
&\lim_{n\rightarrow \infty}(g_i^n,0)\cdot x_0=[g_i^\infty,0], \\
&\lim_{n\rightarrow \infty}(a^n,0)\cdot x_0=[a^\infty,0],
\end{align*}
in $X\cup\partial X$, and 
\begin{align*}
&\lim_{n\rightarrow \infty}(g_i^n,0)*y_0=[g_i^\infty,\frac{\pi}{4}], \\
&\lim_{n\rightarrow \infty}(a^n,0)*y_0=[a^\infty,0],
\end{align*}
in $Y\cup\partial Y$.
Hence any map from $\partial X$ to $\partial Y$ obtained 
as a continuously extension of the quasi-isometry 
$G\cdot x_0\rightarrow G*y_0$ $(g\cdot x_0\rightarrow g*y_0)$ 
must send $[g_i^\infty,0]$ to $[g_i^\infty,\frac{\pi}{4}]$ and 
fix $[a^\infty,0]$.
However, this is incompatible with continuously at $[a^\infty,0]$, 
because $[g_i^\infty,0]\rightarrow [a^\infty,0]$ as $i\rightarrow \infty$ 
(\cite[p.187]{BR}).

Here in this example, we note that 
\begin{enumerate}
\item[(a)] the point $a^i\cdot x_0$ is in the geodesic segment from $x_0$ to $g^i\cdot x_0$ in $X$, 
i.e., $a^i\cdot x_0\in [x_0,g^i\cdot x_0]$ in $X$ for any $i\in\N$ and 
\item[(b)] the distance between the point $a^i*y_0$ and the geodesic segment from $y_0$ to $g^i*y_0$ 
is unbounded for $i\in\N$ in $Y$, i.e., 
there does {\it not} exist a constant $M>0$ such that 
$d(a^i*y_0, [y_0,g^i*y_0])\le M$ for any $i\in\N$ in $Y$.
\end{enumerate}

Based on this observation, we consider a condition.

We suppose that a group $G$ acts geometrically on two CAT(0) spaces $X$ and $Y$.
Let $x_0\in X$ and $y_0\in Y$.
Then we define the condition $(*)$ as follows:
\begin{enumerate}
\item[$(*)$] 
There exist constants $N>0$ and $M>0$ such that $GB(x_0,N)=X$, $GB(y_0,M)=Y$ and 
for any $g,a\in G$, 
if $[x_0,gx_0] \cap B(ax_0,N)\neq\emptyset$ in $X$ then 
$[y_0,gy_0] \cap B(ay_0,M)\neq\emptyset$ in $Y$.
\end{enumerate}

We first prove the following theorem in Section~3.

\begin{Theorem}\label{Thm1}
Suppose that a group $G$ acts geometrically on two CAT(0) spaces $X$ and $Y$.
Let $x_0\in X$ and $y_0\in Y$.
If the condition~$(*)$ holds, then 
there exists a $G$-equivariant homeomorphism of 
the boundaries $\partial X$ and $\partial Y$ 
as a continuous extension of the quasi-isometry 
$\phi:Gx_0\rightarrow Gy_0$ defined by $\phi(gx_0)=gy_0$.
\end{Theorem}

Next, we consider the following condition~$(**)$.

We suppose that a group $G$ acts geometrically on two CAT(0) spaces $X$ and $Y$.
Let $x_0\in X$ and $y_0\in Y$.
Then we define the condition $(**)$ as follows:
\begin{enumerate}
\item[$(**)$] 
For any sequence $\{g_i\,|\,i\in \N\} \subset G$, 
the sequence $\{g_i x_0\,|\,i\in \N\}$ is a Cauchy sequence in $X\cup \partial X$ if and only if 
the sequence $\{g_i y_0\,|\,i\in \N\}$ is a Cauchy sequence in $Y\cup \partial Y$.
\end{enumerate}

Then we show the following theorem in Section~4.

\begin{Theorem}\label{Thm2}
Suppose that a group $G$ acts geometrically on two CAT(0) spaces $X$ and $Y$.
Let $x_0\in X$ and $y_0\in Y$.
The condition~$(**)$ holds if and only if 
there exists a $G$-equivariant homeomorphism of 
the boundaries $\partial X$ and $\partial Y$ 
as a continuous extension of the quasi-isometry 
$\phi:Gx_0\rightarrow Gy_0$ defined by $\phi(gx_0)=gy_0$.
\end{Theorem}

There are problems of {\it rigidity} in group actions (see Section~10).

In this paper, 
a CAT(0) group $G$ is said to be {\it (boundary-)rigid}, 
if $G$ determines its ideal boundary up to homeomorphisms, i.e., 
all boundaries of CAT(0) spaces on which $G$ acts geometrically 
are homeomorphic.
Also a CAT(0) group $G$ is said to be {\it equivariant (boundary) rigid}, 
if $G$ determines its ideal boundary 
by the equivariant homeomorphisms as above, i.e., 
if for any two CAT(0) spaces $X$ and $Y$ on which $G$ acts geometrically 
it obtain a $G$-equivariant homeomorphism of 
the boundaries $\partial X$ and $\partial Y$ 
as a continuous extension of the quasi-isometry 
$\phi:Gx_0\rightarrow Gy_0$ defined by $\phi(gx_0)=gy_0$, 
where $x_0\in X$ and $y_0\in Y$.

As an application of Theorem~\ref{Thm1}, 
we introduce some examples of equivariant rigid CAT(0) groups in Section~5.
In particular, any group of the form 
$$\Z^{n_1}*\dots *\Z^{n_k}*A_1*\dots * A_l$$ 
where $n_i\in \N$ and each $A_j$ is a finite group 
is an equivariant rigid CAT(0) group.
(Here we note that 
there is a recent research by G.~C.~Hruska \cite{Hr} 
on CAT(0) spaces with isolated flats.)

Also as an application of Theorem~\ref{Thm2}, 
we introduce an example of non equivariant rigid CAT(0) groups in Section~6.
In particular, we show that every CAT(0) group of the form $G=F \times H$ 
where $F$ is a free group of rank $n\ge 2$ and $H$ is an infinite CAT(0) group, 
is non equivariant rigid.

In Section~7, we provide some remarks and questions 
on (equivariant) rigidity of boundaries of CAT(0) groups.
Here the following natural open problem is important.

\begin{Problem}
If $G_1$ and $G_2$ are (equivariant) rigid CAT(0) groups, then is $G_1*G_2$ also?
\end{Problem}

In Section~8, we investigate Coxeter groups acting on CAT(0) spaces as reflection groups.

Now we consider a {\it cocompact discrete reflection group} $G$ of a CAT(0) space $X$.
(Here definitions and details of 
{\it reflections} and {\it cocompact discrete reflection groups} of geodesic spaces 
are found in \cite{Ho5} and \cite{Ho6}.)
Then the group $G$ becomes a Coxeter group.

In this paper, 
a Coxeter group $W$ is said to be {\it equivariant rigid as a reflection group}, if 
for any two CAT(0) spaces $X$ and $Y$ on which $W$ acts geometrically 
as $W$ becomes a cocompact discrete reflection group of $X$ and $Y$, 
it obtain a $W$-equivariant homeomorphism of 
the boundaries $\partial X$ and $\partial Y$ 
as a continuous extension of the quasi-isometry 
$\phi:Wx_0\rightarrow Wy_0$ defined by $\phi(wx_0)=wy_0$, 
where $x_0\in X$ and $y_0\in Y$.

Then we show the following theorem.

\begin{Theorem}\label{Thm3}
The following statements hold.
\begin{enumerate}
\item[${\rm (i)}$] If Coxeter groups $W_1$ and $W_2$ are equivariant rigid 
as reflection groups, then $W_1*W_2$ is also.
\item[${\rm (ii)}$] For a Coxeter group $W=W_A*_{W_{A\cap B}}W_B$ where $W_{A\cap B}$ is finite, 
if $W$ determines its Coxeter system up to isomorphism, and 
if $W_A$ and $W_B$ are equivariant rigid as reflection groups then $W$ is also, 
where $W_T$ is the parabolic subgroup of $W$ generated by $T$.
\end{enumerate}
\end{Theorem}

\vspace*{3mm}

Finally, by observing examples and applications of Theorems~\ref{Thm1} and \ref{Thm2},
the following conjecture arises.

\begin{Conjecture}
The group $G=(F_2\times \Z)*\Z_2$ 
will be a non-rigid CAT(0) group with uncountably many boundaries.
\end{Conjecture}

We explain where this conjecture comes from in Section~9 
and introduce problems of rigidity on group actions in Section~10.

\section{CAT(0) spaces and their boundaries}

Details of CAT(0) spaces and their boundaries 
are found in \cite{ABN}, \cite{BH}, \cite{GO}, \cite{GH} and \cite{Sw}.

A proper geodesic space $(X,d_X)$ is called a {\it CAT(0) space}, 
if the ``CAT(0)-inequality'' holds 
for all geodesic triangles $\triangle$ and 
for all choices of two points $x$ and 
$y$ in $\triangle$.
Here the ``CAT(0)-inequality'' is defined as follows: 
Let $\triangle$ be a geodesic triangle in $X$.
A {\it comparison triangle} for $\triangle$ is 
a geodesic triangle $\triangle'$ in the Euclidean plain $\R^2$
with same edge lengths as $\triangle$.
Choose two points $x$ and $y$ in $\triangle$. 
Let $x'$ and $y'$ denote the corresponding points in $\triangle'$.
Then the inequality $$d_X(x,y) \le d_{\R^2}(x',y')$$ 
is called the {\it CAT(0)-inequality}, 
where $d_{\R^2}$ is the natural metric on $\R^2$.

Every proper CAT(0) space can be compactified by 
adding its ``boundary''.
Let $(X,d_X)$ be a proper CAT(0) space, 
and let $\mathcal{R}$ be the set of all geodesic rays in $X$.
We define an equivalence relation $\sim$ in $\mathcal{R}$ as follows:
For geodesic rays $\xi,\zeta:[0,\infty)\rightarrow X$, 
$$ \xi\sim\zeta \iff 
\Image \xi\subset B(\Image\zeta,N) \ \text{for some $N\ge0$},$$
where $B(A,N):=\{x\in X\,|\, d_X(x,A)\le N\}$ for $A\subset X$.
Then the {\it boundary} $\partial X$ of $X$ is 
defined as 
$$ \partial X = \mathcal{R}/\sim. $$
For each geodesic ray $\xi\in\mathcal{R}$, 
the equivalence class of $\xi$ is denoted by $\xi(\infty)$.

It is known that for each $\alpha\in\partial X$ and each $x_0 \in X$, 
there exists a unique geodesic ray $\xi_{\alpha}:[0,\infty)\rightarrow X$ 
such that $\xi_{\alpha}(0)=x_0$ and $\xi_{\alpha}(\infty)=\alpha$.
Thus we can identify the boundary $\partial X$ of $X$ as 
the set of all geodesic rays $\xi$ with $\xi(0)=x_0$.

Let $(X,d_X)$ be a proper CAT(0) space and let $x_0\in X$.
We define a topology on $X\cup \partial X$ as follows: 
\begin{enumerate}
\item[(1)] $X$ is an open subspace of $X\cup\partial X$. 
\item[(2)] Let $\alpha\in\partial X$ and let $\xi_{\alpha}$ be the geodesic ray such that 
$\xi_{\alpha}(0)=x_0$ and $\xi_{\alpha}(\infty)=\alpha$.
For $r>0$ and $\epsilon>0$, we define 
$$ U_{X \cup \partial X}(\alpha;r,\epsilon)=
\{ x \in X \cup \partial X \,|\,
x \not\in B(x_0,r),\ d_X(\xi_{\alpha}(r),\xi_{x}(r))<\epsilon \}, $$
where $\xi_{x}:[0,d_X(x_0,x)]\rightarrow X$ is 
the geodesic (segment or ray) from $x_0$ to $x$.
Let $\epsilon_0>0$ be a constant.
Then the set 
$$\{U_{X\cup\partial X}(\alpha;r,\epsilon_0)\,|\,r>0 \} $$
is a neighborhood basis for $\alpha$ in $X\cup\partial X$.
\end{enumerate}
Here it is known that the topology on $X\cup \partial X$ 
is not dependent on the basepoint $x_0 \in X$ and 
$X\cup\partial X$ is a metrizable compactification of $X$.

Also 
for $\alpha\in\partial X$ and the geodesic ray $\xi_{\alpha}$ with 
$\xi_{\alpha}(0)=x_0$ and $\xi_{\alpha}(\infty)=\alpha$ 
and for $r>0$ and $\epsilon>0$, 
we define 
$$ U'_{X \cup \partial X}(\alpha;r,\epsilon)=
\{ x \in X \cup \partial X \,|\,
x \not\in B(x_0,r),\ d_X(\xi_{\alpha}(r),\Image\xi_{x})<\epsilon \}, $$
where $\xi_{x}:[0,d(x_0,x)]\rightarrow X$ is 
the geodesic (segment or ray) from $x_0$ to $x$.
Let $\epsilon_0>0$ be a constant.
Then the set 
$$\{U'_{X\cup\partial X}(\alpha;r,\epsilon_0)\,|\,r>0 \} $$
is also a neighborhood basis for $\alpha$ in $X\cup\partial X$ 
(cf.\ \cite[Lemma~4.2]{Ho000}).

Suppose that a group $G$ acts on a proper CAT(0) space $X$ by isometries.
For each element $g \in G$ and 
each geodesic ray $\xi:[0,\infty)\rightarrow X$, 
a map $g \xi:[0,\infty)\rightarrow X$ 
defined by $(g\xi)(t):=g(\xi(t))$ is also a geodesic ray.
For two geodesic rays $\xi$ and $\xi'$, 
if $\xi(\infty)=\xi'(\infty)$ 
then $g\xi(\infty)=g\xi'(\infty)$. 
Thus $g$ induces a homeomorphism of $\partial X$, 
and $G$ acts on $\partial X$ by homeomorphisms.
Here we note that 
if a sequence $\{x_i\,|\,i\in \N\}\subset X$ 
converges to $\alpha\in\partial X$ in $X\cup\partial X$, then 
for any $g\in G$, 
the sequence $\{gx_i\,|\,i\in \N\}\subset X$ 
converges to $g\alpha\in\partial X$ in $X\cup\partial X$.

\begin{Definition}
Let $(X,d_X)$ be a proper CAT(0) space and 
let $\{x_i\,|\,i\in\N\}\subset X$ be an unbounded sequence in $X$.
In this paper, we say that 
the sequence $\{x_i\,|\,i\in\N\}$ is a {\it Cauchy sequence in $X\cup\partial X$}, 
if there exists $\epsilon_0>0$ such that for any $r>0$, 
there is a number $i_0\in\N$ as 
$$ x_i \in U_{X\cup\partial X}(x_{i_0};r,\epsilon_0) $$
for any $i\ge i_0$.
Here 
$$ U_{X\cup\partial X}(x_{i_0};r,\epsilon_0)= 
\{ x \in X \,|\, x \not\in B(x_0,r),\ d_X(\xi_{x_{i_0}}(r),\xi_{x}(r))<\epsilon \}, $$
where $\xi_{z}$ is the geodesic segment from $x_0$ to $z$ in $X$.
\end{Definition}

We show the following lemma used later.

\begin{Lemma}\label{Lem1}
Let $(X,d_X)$ be a proper CAT(0) space and 
let $\{x_i\,|\,i\in\N\}\subset X$ be an unbounded sequence in $X$.
Then the sequence $\{x_i\,|\,i\in\N\}$ is a Cauchy sequence in $X\cup\partial X$ defined above 
if and only if 
the sequence $\{x_i\,|\,i\in\N\}$ converges to some point $\alpha\in\partial X$ in $X\cup\partial X$.
\end{Lemma}

\begin{proof}
We first show that 
if the sequence $\{x_i\,|\,i\in\N\}$ converges to some point $\alpha\in\partial X$ in $X\cup\partial X$, 
then $\{x_i\,|\,i\in\N\}$ is a Cauchy sequence in $X\cup\partial X$ defined above.

Suppose that $\{x_i\,|\,i\in\N\}$ converges to $\alpha\in\partial X$ in $X\cup\partial X$.
Let $\epsilon_0>0$.
Since the set 
$$\{U_{X\cup\partial X}(\alpha;r,\frac{\epsilon_0}{2})\,|\,r>0\}$$ 
is a neighborhood basis for $\alpha$ in $X\cup\partial X$, for each $r>0$, 
there exists a number $i_0\in\N$ such that 
$$x_i\in U_{X\cup\partial X}(\alpha;r,\frac{\epsilon_0}{2})$$ 
for any $i\ge i_0$.
Then for any $i\ge i_0$, 
\begin{align*}
d_X(\xi_{x_{i_0}}(r),\xi_{x}(r)) &\le 
d_X(\xi_{x_{i_0}}(r),\xi_{\alpha}(r)) + d_X(\xi_{\alpha}(r),\xi_{x}(r)) \\
&\le \frac{\epsilon_0}{2}+\frac{\epsilon_0}{2} \\
&=\epsilon_0.
\end{align*}
Hence 
$x_i \in U_{X\cup\partial X}(x_{i_0};r,\epsilon_0)$ for any $i\ge i_0$.
Thus the sequence $\{x_i\,|\,i\in\N\}$ is a Cauchy sequence in $X\cup\partial X$.

Next, we show that 
if $\{x_i\,|\,i\in\N\}$ is a Cauchy sequence in $X\cup\partial X$ defined above, 
then 
$\{x_i\,|\,i\in\N\}$ converges to some point $\alpha\in\partial X$ in $X\cup\partial X$.

Suppose that $\{x_i\,|\,i\in\N\}$ is a Cauchy sequence in $X\cup\partial X$.
Since the set $\{x_i\,|\,i\in\N\}$ is unbounded in $X$, 
there exists a limit point $\alpha\in \Cl\{x_i\,|\,i\in\N\} \cap \partial X$.
Here there exists a subsequence $\{x_{i_j}\,|\,j\in \N\}\subset \{x_i\,|\,i\in\N\}$ 
which converges to $\alpha$ in $X\cup\partial X$.

Then we show that 
the sequence $\{x_i\,|\,i\in\N\}$ converges to the point $\alpha\in\partial X$ in $X\cup\partial X$.

Since $\{x_i\,|\,i\in\N\}$ is a Cauchy sequence in $X\cup\partial X$, 
there exists $\epsilon_0>0$ such that for any $r>0$, 
there is a number $i_0\in\N$ as 
$x_i \in U_{X\cup\partial X}(x_{i_0};r,\epsilon_0)$ for any $i\ge i_0$, i.e., 
$d_X(\xi_{x_{i_0}}(r),\xi_{x_i}(r))\le \epsilon_0$ for any $i\ge i_0$.
Also since the subsequence $\{x_{i_j}\,|\,j\in \N\}$ converges to $\alpha$ in $X\cup\partial X$, 
there exists $i_{j_0}\ge i_0$ such that 
$x_{i_{j_0}}\in U_{X\cup\partial X}(\alpha;r,1)$, i.e., 
$d_X(\xi_{x_{i_{j_0}}}(r),\xi_{\alpha}(r))\le 1$.
Then for any $i\ge i_{j_0}$, 
\begin{align*}
d_X(\xi_{x_i}(r),\xi_{\alpha}(r)) &\le 
d_X(\xi_{x_i}(r),\xi_{x_{i_0}}(r)) + d_X(\xi_{x_{i_0}}(r),\xi_{x_{i_{j_0}}}(r)) 
+ d_X(\xi_{x_{i_{j_0}}}(r),\xi_{\alpha}(r)) \\
&\le \epsilon_0 + \epsilon_0 + 1 \\
&= 2\epsilon_0+1,
\end{align*}
since $i\ge i_{j_0}\ge i_0$.
Hence for any $r>0$, 
there exists a number $i_{j_0}\in\N$ such that for any $i\ge i_{j_0}$, 
$$x_i\in U_{X\cup\partial X}(\alpha;r,2\epsilon_0+1),$$ 
where $2\epsilon_0+1$ is a constant.
Thus the sequence $\{x_i\,|\,i\in\N\}$ converges 
to the point $\alpha\in\partial X$ in $X\cup\partial X$.
\end{proof}

\section{Proof of Theorem~\ref{Thm1}}

In this section, we prove Theorem~\ref{Thm1}.

We suppose that a group $G$ acts geometrically on two CAT(0) spaces $(X,d_X)$ and $(Y,d_Y)$.
Let $x_0\in X$ and $y_0\in Y$.
Now we suppose that the condition $(*)$ holds; that is, 
\begin{enumerate}
\item[$(*)$] 
there exist constants $N>0$ and $M>0$ such that $GB(x_0,N)=X$, $GB(y_0,M)=Y$ and 
for any $g,a\in G$, 
if $[x_0,gx_0] \cap B(ax_0,N)\neq\emptyset$ in $X$ then 
$[y_0,gy_0] \cap B(ay_0,M)\neq\emptyset$ in $Y$.
\end{enumerate}

Our goal in this section is to show that 
the quasi-isometry $\phi:Gx_0\rightarrow Gy_0$ defined by $\phi(gx_0)=gy_0$ 
continuously extends to a $G$-equivariant homeomorphism 
of the boundaries $\partial X$ and $\partial Y$.

Since the map $\phi:Gx_0\rightarrow Gy_0$ defined by $\phi(gx_0)=gy_0$ 
is a quasi-isometry (cf.\ \cite[p.138]{BH}, \cite{Gr}, \cite{Gr0}), 
there exist constants $\lambda>0$ and $C>0$ such that 
$$ \frac{1}{\lambda} d_Y(gy_0,hy_0)-C \le d_X(gx_0,hx_0) \le \lambda d_Y(gy_0,hy_0)+C $$
for any $g,h\in G$.

We first show the following.

\begin{Proposition}\label{Prop2-1}
Let $\{g_i\}\subset G$ be a sequence.
If $\{g_ix_0\}\subset X$ is a Cauchy sequence in $X\cup \partial X$ defined in Section~2, 
then 
$\{g_iy_0\}\subset Y$ is also a Cauchy sequence in $Y\cup \partial Y$.
\end{Proposition}

\begin{proof}
Let $\{g_i\}\subset G$.
Suppose that $\{g_ix_0\}\subset X$ is a Cauchy sequence in $X\cup \partial X$.

To prove that $\{g_iy_0\}\subset Y$ is a Cauchy sequence in $Y\cup \partial Y$, 
we show that 
there exists $M'>0$ such that for any $R>0$, 
there is $i_0\in \N$ as 
$$ g_iy_0 \in U_{Y\cup\partial Y}(g_{i_0}y_0;R,M') $$
for any $i\ge i_0$.

Let $M'=\lambda(2N+1)+2M+C$ and let $R>0$.

Since $\{g_ix_0\}\subset X$ is a Cauchy sequence in $X\cup \partial X$, 
for $r=\lambda(R+C+M)+N$, 
there exists $i_0\in \N$ such that 
$$ g_ix_0 \in U_{X\cup\partial X}(g_{i_0}x_0;r,1) $$
for any $i\ge i_0$.

Then 
$$d_X(x_0,g_{i_0}x_0)\ge r, \ d_X(x_0,g_ix_0)\ge r, \ \text{and}\ 
d_X(\xi_{g_{i_0}x_0}(r),\xi_{g_ix_0}(r)) \le 1,$$
where $\xi_{g_{i_0}x_0}$ is the geodesic from $x_0$ to $g_{i_0}x_0$ and 
$\xi_{g_ix_0}$ is the geodesic from $x_0$ to $g_ix_0$ in $X$.

Since $GB(x_0,N)=X$, 
there exist $a,b\in G$ such that 
$d_X(ax_0,\xi_{g_{i_0}x_0}(r))\le N$ and $d_X(bx_0,\xi_{g_ix_0}(r))\le N$.
Then 
$$ [x_0,g_{i_0}x_0]\cap B(ax_0,N)\neq \emptyset \ \text{and} \ 
[x_0,g_ix_0]\cap B(bx_0,N)\neq \emptyset.$$
Hence by the condition $(*)$, 
$$ [y_0,g_{i_0}y_0]\cap B(ay_0,M)\neq \emptyset \ \text{and} \ 
[y_0,g_iy_0]\cap B(by_0,M)\neq \emptyset.$$
Thus 
$$\xi_{g_{i_0}y_0}(r'_0)\in [y_0,g_{i_0}y_0]\cap B(ay_0,M) \ \text{and} \ 
\xi_{g_iy_0}(r')\in [y_0,g_iy_0]\cap B(by_0,M)$$ 
for some $r'_0>0$ and $r'>0$.

To obtain that for any $i\ge i_0$, 
$$ g_iy_0 \in U_{Y\cup\partial Y}(g_{i_0}y_0;R,M'), $$
we show that 
$$ r'_0 \ge R, \ r'\ge R \ \text{and}\ 
d_Y(\xi_{g_{i_0}y_0}(r'_0),\xi_{g_iy_0}(r'))\le M'.$$

First, 
\begin{align*}
r'_0 &= d_Y(y_0,\xi_{g_{i_0}y_0}(r'_0)) \\
&\ge d_Y(y_0,ay_0)-M \\
&\ge \frac{1}{\lambda}d_X(x_0,ax_0)-C-M \\
&\ge \frac{1}{\lambda}(r-N)-C-M \\
&=R,
\end{align*}
because $d_X(x_0,ax_0)\ge r-N$ and $r=\lambda(R+C+M)+N$.

Similary, $r'\ge R$, 
because $d_X(x_0,bx_0)\ge r-N$ and $r=\lambda(R+C+M)+N$.

Also, 
\begin{align*}
d_Y(\xi_{g_{i_0}y_0}(r'_0),\xi_{g_iy_0}(r')) 
&\le d_Y(ay_0,by_0)+2M \\
&\le (\lambda d_X(ax_0,bx_0)+C)+2M \\
&\le \lambda(d_X(\xi_{g_{i_0}x_0}(r),\xi_{g_ix_0}(r))+2N)+C+2M \\
&\le \lambda(1+2N)+C+2M \\
&= M',
\end{align*}
because $d_X(\xi_{g_{i_0}x_0}(r),\xi_{g_ix_0}(r))\le 1$ and 
$M'=\lambda(2N+1)+2M+C$.

Thus 
$$ r'_0 \ge R, \ r'\ge R \ \text{and}\ 
d_Y(\xi_{g_{i_0}y_0}(r'_0),\xi_{g_iy_0}(r'))\le M'.$$
Hence 
$$ 
d_Y(\xi_{g_{i_0}y_0}(R),\xi_{g_iy_0}(R))\le 
d_Y(\xi_{g_{i_0}y_0}(r'_0),\xi_{g_iy_0}(r')) \le M',$$
since $Y$ is a CAT(0) space.
Also we obtain that 
$$d_Y(y_0,g_{i_0}y_0)\ge R \ \text{and} \ d_Y(y_0,g_iy_0)\ge R,$$
because $r'_0 \ge R$ and $r'\ge R$.

Thus 
$$ g_iy_0 \in U_{Y\cup\partial Y}(g_{i_0}y_0;R,M') $$
for any $i\ge i_0$.
Hence we obtain that 
$\{g_iy_0\}\subset Y$ is a Cauchy sequence in $Y\cup \partial Y$.
\end{proof}

Then we can define a map 
$\bar{\phi}:\partial X\rightarrow \partial Y$ 
as a continuous extension of the quasi-isometry 
$\phi:Gx_0\rightarrow Gy_0$ defined by $\phi(gx_0)=gy_0$ as follows:
For each $\alpha\in \partial X$, 
there exists a sequence $\{g_ix_0\}\subset Gx_0\subset X$ 
which converges to $\alpha$ in $X\cup\partial X$.
Then the sequence $\{g_ix_0\}\subset X$ is a Cauchy sequence in $X\cup \partial X$ 
by Lemma~\ref{Lem1}.
By Proposition~\ref{Prop2-1}, 
the sequence $\{g_iy_0\}\subset Y$ is also a Cauchy sequence in $Y\cup\partial Y$.
Hence by Lemma~\ref{Lem1}, 
the sequence $\{g_iy_0\}\subset Y$ converges to 
some point $\bar{\alpha}\in\partial Y$ in $Y\cup\partial Y$.
Then we define $\bar{\phi}(\alpha)=\bar{\alpha}$.

\begin{Proposition}\label{Prop2-2}
The map $\bar{\phi}:\partial X\rightarrow \partial Y$ is well-defined.
\end{Proposition}

\begin{proof}
Let $\alpha\in \partial X$ and 
let $\{g_ix_0\},\{h_ix_0\} \subset Gx_0\subset X$ be two sequences 
which converge to $\alpha$ in $X\cup\partial X$.
As the argument above, by Lemma~\ref{Lem1} and Proposition~\ref{Prop2-1}, 
the sequence $\{g_iy_0\}\subset Y$ converges to 
some point $\bar{\alpha}\in\partial Y$ and 
the sequence $\{h_iy_0\}\subset Y$ converges to 
some point $\bar{\beta}\in\partial Y$ in $Y\cup\partial Y$.
Then we show that $\bar{\alpha}=\bar{\beta}$.

Here we can consider a sequence $\{\tilde{g}_jx_0\,|\,j\in\N\}\subset Gx_0\subset X$ 
such that 
$$\{\tilde{g}_jx_0\,|\,j\in\N\}=\{g_ix_0\,|\,i\in\N\}\cup \{h_ix_0\,|\,i\in\N\}$$
and the sequence $\{\tilde{g}_jx_0\}$ converges to $\alpha$ in $X\cup\partial X$.
Then the sequence $\{\tilde{g}_jx_0\}$ is a Cauchy sequence in $X\cup\partial X$ and 
the sequence $\{\tilde{g}_jy_0\}$ is also in $Y\cup\partial Y$ by Proposition~\ref{Prop2-1}.
Hence the sequence $\{\tilde{g}_jy_0\}$ converges to 
some point $\bar{\gamma}\in\partial Y$ in $Y\cup\partial Y$.
Here we note that 
the two sequences $\{g_iy_0\}$ and $\{h_iy_0\}$ are subsequences of $\{\tilde{g}_jy_0\}$.
Hence we obtain that $\bar{\alpha}=\bar{\beta}=\bar{\gamma}$.

Thus the map $\bar{\phi}:\partial X\rightarrow \partial Y$ defined as above 
is well-defined.
\end{proof}

Next, we show the following.

\begin{Proposition}\label{Prop2-3}
The map $\bar{\phi}:\partial X\rightarrow \partial Y$ is surjective.
\end{Proposition}

\begin{proof}
Let $\bar{\alpha}\in \partial Y$.
There exists a sequence $\{g_iy_0\}\subset Gy_0 \subset Y$ 
which converges to $\bar{\alpha}$ in $Y\cup \partial Y$.
Then we consider the set $\{g_ix_0\,|\,i\in \N\}$ 
which is an unbounded subset of $X$.
Here 
$$ \Cl\{g_ix_0\,|\,i\in \N\} \cap \partial X \neq \emptyset,$$
and there exists a subsequence $\{g_{i_j}x_0\,|\,j\in\N\}\subset \{g_ix_0\}$ 
which converges to some point $\alpha \in \partial X$.
Then the sequence $\{g_{i_j}y_0\}$ converges to $\bar{\alpha}$ in $Y\cup \partial Y$, 
because $\{g_{i_j}y_0\}$ is a subsequence of the sequence $\{g_iy_0\}$ 
which converges to $\bar{\alpha}$ in $Y\cup \partial Y$.
Hence $\bar{\phi}(\alpha)=\bar{\alpha}$ 
by the definition of the map $\bar{\phi}$.
Thus the map $\bar{\phi}:\partial X\rightarrow \partial Y$ is surjective.
\end{proof}

Here we provide a lemma.

\begin{Lemma}\label{Lem2-4}
For any $\tilde{N}\ge N$, 
there exists $\tilde{M}>0$ such that $GB(y_0,\tilde{M})=Y$ and 
for any $g,a\in G$, 
if $[x_0,gx_0] \cap B(ax_0,\tilde{N})\neq\emptyset$ in $X$ then 
$[y_0,gy_0] \cap B(ay_0,\tilde{M})\neq\emptyset$ in $Y$.
\end{Lemma}

\begin{proof}
For $\tilde{N}\ge N$, 
we put $\tilde{M}=\lambda(N+\tilde{N})+C+M$.

Let $g,a\in G$ 
as $[x_0,gx_0] \cap B(ax_0,\tilde{N})\neq\emptyset$ in $X$.
Then there exists a point $x_1\in [x_0,gx_0] \cap B(ax_0,\tilde{N})$.
Since $GB(x_0,N)=X$, 
there exists $a'\in G$ such that $x_1 \in B(a'x_0,N)$.
Then $x_1 \in [x_0,gx_0]\cap B(a'x_0,N)$ and 
$[x_0,gx_0]\cap B(a'x_0,N)\neq\emptyset$ in $X$.
By the condition~$(*)$, 
$[y_0,gy_0]\cap B(a'y_0,M)\neq\emptyset$ in $Y$.
Hence $d_Y(a'y_0,[y_0,gy_0])\le M$.
Here we note that 
\begin{align*}
d_Y(a'y_0,ay_0) &\le \lambda d_X(a'x_0,ax_0)+C \\
&\le \lambda(d_X(a'x_0,x_1)+d_X(x_1,ax_0))+C \\
&\le \lambda(N+\tilde{N})+C.
\end{align*}
Hence 
\begin{align*}
d_Y(ay_0,[y_0,gy_0]) &\le d_Y(ay_0,a'y_0)+d_Y(a'y_0,[y_0,gy_0]) \\
&\le \lambda(N+\tilde{N})+C+M \\
&=\tilde{M}.
\end{align*}
Thus we obtain that 
$[y_0,gy_0] \cap B(ay_0,\tilde{M})\neq\emptyset$ in $Y$.
\end{proof}

Let $\tilde{N}=2N$.
By Lemma~\ref{Lem2-4}, 
there exists $\tilde{M}>0$ such that $GB(y_0,\tilde{M})=Y$ and 
for any $g,a\in G$, 
if $[x_0,gx_0] \cap B(ax_0,\tilde{N})\neq\emptyset$ in $X$ then 
$[y_0,gy_0] \cap B(ay_0,\tilde{M})\neq\emptyset$ in $Y$.

Here we show the following technical lemma.

\begin{Lemma}\label{Lem2-5}
Let $\alpha\in\partial X$ and let $\xi_{\alpha}:[0,\infty)\rightarrow X$ 
be the geodesic ray in $X$ such that 
$\xi_{\alpha}(0)=x_0$ and $\xi_{\alpha}(\infty)=\alpha$.
Let $\{g_ix_0\}\subset Gx_0 \subset X$ be a sequence 
which converges to $\alpha$ in $X\cup\partial X$ 
such that $d_X(g_ix_0,\xi_{\alpha}(i))\le N$ for any $i\in\N$ 
(since $GB(x_0,N)=X$, we can take such a sequence).
Then 
\begin{enumerate}
\item[$(1)$] $d_X(g_ix_0,[x_0,g_jx_0]) \le \tilde{N}$ for any $i,j\in \N$ with $i<j$,
\item[$(2)$] $d_Y(g_iy_0,[y_0,g_jy_0]) \le \tilde{M}$ for any $i,j\in \N$ with $i<j$,
\item[$(3)$] $d_Y(g_iy_0,\Image \xi_{\bar{\alpha}}) \le \tilde{M}+1$ for any $i\in\N$,
\item[$(4)$] $d_X(g_ix_0,g_{i+1}x_0)\le 2N+1$ for any $i\in\N$,
\item[$(5)$] $d_Y(g_iy_0,g_{i+1}y_0)\le \lambda(2N+1)+C$ for any $i\in\N$, and 
\item[$(6)$] $\Image\xi_{\bar{\alpha}} \subset 
\bigcup\{B(g_iy_0,3(\tilde{M}+1)+\lambda(2N+1)+C)\,|\,i\in\N \}$.
\end{enumerate}
Here $\bar{\alpha}=\bar{\phi}(\alpha)$ and 
$\xi_{\bar{\alpha}}:[0,\infty)\rightarrow Y$ is the geodesic ray in $Y$ such that 
$\xi_{\bar{\alpha}}(0)=y_0$ and $\xi_{\bar{\alpha}}(\infty)=\bar{\alpha}$.
\end{Lemma}

\begin{proof}
(1) 
For any $i,j\in \N$ with $i<j$, 
\begin{align*}
d_X(g_ix_0,[x_0,g_jx_0]) 
&\le d_X(g_ix_0,\xi_{\alpha}(i))+d_X(\xi_{\alpha}(i),[x_0,g_jx_0]) \\
&\le N+N =2N \\
&=\tilde{N},
\end{align*}
where we obtain the inequality $d_X(\xi_{\alpha}(i),[x_0,g_jx_0])\le N$, 
since $d_X(g_jx_0,\xi_{\alpha}(j))\le N$, $i<j$ and $X$ is a CAT(0) space.

(2) 
By Lemma~\ref{Lem2-4} and the definition of $\tilde{M}$, 
we obtain that 
$d_Y(g_iy_0,[y_0,g_jy_0]) \le \tilde{M}$ for any $i,j\in \N$ with $i<j$ from (1).

(3) 
We note that 
the sequence $\{g_iy_0\}$ converges to $\bar{\alpha}$ 
by the definition of the map $\bar{\phi}:\partial X\rightarrow \partial Y$.

Let $i\in\N$ and let $R=d_Y(y_0,g_iy_0)$.
Since the sequence $\{g_jy_0\}$ converges to $\bar{\alpha}$, 
there exists $j_0\in N$ such that 
$$d_Y(\xi_{\bar{\alpha}}(R),\xi_{g_jy_0}(R))<1$$
for any $j\ge j_0$, 
because the set 
$$\{U_{Y\cup\partial Y}(\bar{\alpha};r,1)\,|\,r>0\}$$ 
defined in Section~2 
is a neighborhood basis for $\bar{\alpha}$ in $Y\cup\partial Y$.

Let $j\in \N$ with $j>i$ and $j>j_0$.
Since $i<j$, we obtain that $d_Y(g_iy_0,[y_0,g_jy_0]) \le \tilde{M}$ by (2).
Hence there exists $r>0$ such that 
$d_Y(g_iy_0,\xi_{g_jy_0}(r)) \le \tilde{M}$.
Here we note that $r \le R$ by \cite[Lemma~4.1]{Ho000}
and we can obtain that 
$$d_Y(\xi_{\bar{\alpha}}(r),\xi_{g_jy_0}(r))<d_Y(\xi_{\bar{\alpha}}(R),\xi_{g_jy_0}(R))<1,$$
since $Y$ is a CAT(0) space.
Then 
\begin{align*}
d_Y(g_iy_0,\Image \xi_{\bar{\alpha}}) 
&\le d_Y(g_iy_0,\xi_{g_jy_0}(r))+d_Y(\xi_{g_jy_0}(r),\Image \xi_{\bar{\alpha}}) \\
&< \tilde{M}+1.
\end{align*}
Hence
$d_Y(g_iy_0,\Image \xi_{\bar{\alpha}}) \le \tilde{M}+1$ for any $i\in\N$.

(4) 
We obtain that $d_X(g_ix_0,g_{i+1}x_0)\le 2N+1$ for any $i\in\N$, 
because 
\begin{align*}
d_X(g_ix_0,g_{i+1}x_0) &\le 
d_X(g_ix_0,\xi_{\alpha}(i))+d_X(\xi_{\alpha}(i),\xi_{\alpha}(i+1))+d_X(\xi_{\alpha}(i+1),g_{i+1}x_0) \\
&\le N+1+N \\
&=2N+1,
\end{align*}
since $d_X(g_ix_0,\xi_{\alpha}(i))\le N$ for any $i\in\N$ 
by the definition of the sequence $\{g_ix_0\}$.

(5) 
Since the map $\phi:Gx_0\rightarrow Gy_0 \ (gx_0\mapsto gy_0)$ is a quasi-isometry, 
we obtain that 
$d_Y(g_iy_0,g_{i+1}y_0)\le \lambda(2N+1)+C$ for any $i\in\N$ by (4).

(6) 
For each $i\in\N$, there exists $r_i>0$ such that 
$d_Y(g_iy_0,\xi_{\bar{\alpha}}(r_i))\le \tilde{M}+1$ by (3).
Then by (5), 
\begin{align*}
d_Y(\xi_{\bar{\alpha}}(r_i),\xi_{\bar{\alpha}}(r_{i+1})) &\le 
d_Y(\xi_{\bar{\alpha}}(r_i),g_iy_0)+d_Y(g_iy_0,g_{i+1}y_0)+
d_Y(g_{i+1}y_0,\xi_{\bar{\alpha}}(r_{i+1})) \\
&\le (\tilde{M}+1)+(\lambda(2N+1)+C)+(\tilde{M}+1) \\
&= 2(\tilde{M}+1)+\lambda(2N+1)+C.
\end{align*}
Hence we obtain that 
$$ \Image\xi_{\bar{\alpha}} \subset 
\bigcup\{B(g_iy_0,3(\tilde{M}+1)+\lambda(2N+1)+C)\,|\,i\in\N \}.$$
\end{proof}

Now we show the following.

\begin{Proposition}\label{Prop2-6}
The map $\bar{\phi}:\partial X\rightarrow \partial Y$ is injective.
\end{Proposition}

\begin{proof}
Let $\alpha,\alpha'\in\partial X$, and 
let $\xi_{\alpha}:[0,\infty)\rightarrow X$ and $\xi_{\alpha'}:[0,\infty)\rightarrow X$ 
be the geodesic rays in $X$ such that 
$\xi_{\alpha}(0)=\xi_{\alpha'}(0)=x_0$, 
$\xi_{\alpha}(\infty)=\alpha$ and $\xi_{\alpha'}(\infty)=\alpha'$.
Let 
$\{g_ix_0\},\{g'_ix_0\}\subset Gx_0 \subset X$ be sequences such that 
$d_X(g_ix_0,\xi_{\alpha}(i))\le N$ and $d_X(g'_ix_0,\xi_{\alpha'}(i))\le N$.
Then 
the sequence $\{g_ix_0\}$ converges to $\alpha$ and 
the sequence $\{g'_ix_0\}$ converges to $\alpha'$ in $X\cup\partial X$.

Let $\bar{\alpha}=\bar{\phi}(\alpha)$ and $\bar{{\alpha'}}=\bar{\phi}(\alpha')$.
Also let $\xi_{\bar{\alpha}}:[0,\infty)\rightarrow Y$ and 
$\xi_{\bar{{\alpha'}}}:[0,\infty)\rightarrow Y$ 
be the geodesic rays in $Y$ such that 
$\xi_{\bar{\alpha}}(0)=\xi_{\bar{{\alpha'}}}(0)=y_0$, 
$\xi_{\bar{\alpha}}(\infty)=\bar{\alpha}$ and 
$\xi_{\bar{{\alpha'}}}(\infty)=\bar{{\alpha'}}$.

Then by Lemma~\ref{Lem2-5}, 
\begin{enumerate}
\item[$(1)$] $d_X(g_ix_0,[x_0,g_jx_0]) \le \tilde{N}$ for any $i,j\in \N$ with $i<j$,
\item[$(2)$] $d_Y(g_iy_0,[y_0,g_jy_0]) \le \tilde{M}$ for any $i,j\in \N$ with $i<j$,
\item[$(3)$] $d_Y(g_iy_0,\Image \xi_{\bar{\alpha}}) \le \tilde{M}+1$ for any $i\in\N$,
\item[$(4)$] $d_X(g_ix_0,g_{i+1}x_0)\le 2N+1$ for any $i\in\N$,
\item[$(5)$] $d_Y(g_iy_0,g_{i+1}y_0)\le \lambda(2N+1)+C$ for any $i\in\N$,
\item[$(6)$] $\Image\xi_{\bar{\alpha}} \subset 
\bigcup\{B(g_iy_0,3(\tilde{M}+1)+\lambda(2N+1)+C)\,|\,i\in\N \}$, 
\end{enumerate}
and 
\begin{enumerate}
\item[$(1')$] $d_X(g'_ix_0,[x_0,g'_jx_0]) \le \tilde{N}$ for any $i,j\in \N$ with $i<j$,
\item[$(2')$] $d_Y(g'_iy_0,[y_0,g'_jy_0]) \le \tilde{M}$ for any $i,j\in \N$ with $i<j$,
\item[$(3')$] $d_Y(g'_iy_0,\Image \xi_{\bar{{\alpha'}}}) \le \tilde{M}+1$ for any $i\in\N$,
\item[$(4')$] $d_X(g'_ix_0,g'_{i+1}x_0)\le 2N+1$ for any $i\in\N$,
\item[$(5')$] $d_Y(g'_iy_0,g'_{i+1}y_0)\le \lambda(2N+1)+C$ for any $i\in\N$,
\item[$(6')$] $\Image\xi_{\bar{{\alpha'}}} \subset 
\bigcup\{B(g'_iy_0,3(\tilde{M}+1)+\lambda(2N+1)+C)\,|\,i\in\N \}$.
\end{enumerate}

To prove that 
the map $\bar{\phi}:\partial X\rightarrow \partial Y$ is injective, 
we show that if $\alpha\neq\alpha'$ then $\bar{\alpha}\neq\bar{{\alpha'}}$.

We suppose that $\alpha\neq\alpha'$.
Then the geodesic rays $\xi_{\alpha}$ and $\xi_{\alpha'}$ are not asymptotic.
Hence for any $t>0$, 
there exists $r_0>0$ such that 
$d_X(\xi_{\alpha}(r_0),\Image \xi_{\alpha'})>t$.
Then for $i_0\in\N$ with $i_0\ge r_0$,
\begin{align*}
d_X(g_{i_0}x_0,\Image \xi_{\alpha'}) 
&\ge 
d_X(\xi_{\alpha}(i_0),\Image \xi_{\alpha'})-d_X(g_{i_0}x_0,\xi_{\alpha}(i_0)) \\
&\ge 
d_X(\xi_{\alpha}(r_0),\Image \xi_{\alpha'})-N \\
&> t-N
\end{align*}
Since $d_X(g'_jx_0,\Image \xi_{\alpha'}) \le N$ for any $j\in \N$, 
we obtain that 
$d_X(g_{i_0}x_0,g'_jx_0) > t-2N$ for any $j\in \N$.
Hence for any $j\in \N$, 
\begin{align*}
d_Y(g_{i_0}y_0,g'_jy_0) &\ge \frac{1}{\lambda}d_X(g_{i_0}x_0,g'_jx_0)-C \\
&> \frac{1}{\lambda}(t-2N)-C.
\end{align*}
Here by $(6')$, 
$$ \Image\xi_{\bar{{\alpha'}}} \subset 
\bigcup\{B(g'_jy_0,3(\tilde{M}+1)+\lambda(2N+1)+C)\,|\,j\in\N \}.$$
Let $j_0\in\N$ such that 
$$d_Y(g_{i_0}y_0,g'_{j_0}y_0)=\min\{d_Y(g_{i_0}y_0,g'_jy_0)\,|\,j\in\N \}.$$
Then 
\begin{align*}
d_Y(g_{i_0}y_0,\Image \xi_{\bar{{\alpha'}}}) &\ge 
\min\{d_Y(g_{i_0}y_0,g'_jy_0)\,|\,j\in\N \}-(3(\tilde{M}+1)+\lambda(2N+1)+C) \\
&= d_Y(g_{i_0}y_0,g'_{j_0}y_0)-(3(\tilde{M}+1)+\lambda(2N+1)+C) \\
&> (\frac{1}{\lambda}(t-2N)-C)-(3(\tilde{M}+1)+\lambda(2N+1)+C),
\end{align*}
since 
$d_Y(g_{i_0}y_0,g'_jy_0)>\frac{1}{\lambda}(t-2N)-C$ for any $j\in \N$ 
by the argument above.

Thus for any $t>0$, 
there exists $i_0\in \N$ such that 
$$ d_Y(g_{i_0}y_0,\Image \xi_{\bar{{\alpha'}}}) > 
(\frac{1}{\lambda}(t-2N)-C)-(3(\tilde{M}+1)+\lambda(2N+1)+C).$$
Here by (3), there exists $R_0>0$ such that 
$$ d_Y(g_{i_0}y_0,\xi_{\bar{\alpha}}(R_0)) \le \tilde{M}+1.$$
Then 
\begin{align*}
d_Y(\xi_{\bar{\alpha}}(R_0),\Image \xi_{\bar{{\alpha'}}}) 
&\ge d_Y(g_{i_0}y_0,\Image \xi_{\bar{{\alpha'}}})-d_Y(g_{i_0}y_0,\xi_{\bar{\alpha}}(R_0)) \\
&> (\frac{1}{\lambda}(t-2N)-C)-(3(\tilde{M}+1)+\lambda(2N+1)+C)-(\tilde{M}+1) \\
&= (\frac{1}{\lambda}(t-2N)-C)-(4(\tilde{M}+1)+\lambda(2N+1)+C).
\end{align*}
Since $t>0$ is an arbitrary large number, 
the two geodesic rays $\xi_{\bar{\alpha}}$ and $\xi_{\bar{{\alpha'}}}$ are not asymptotic 
and $\bar{\alpha}\neq\bar{{\alpha'}}$.

Therefore, 
the map $\bar{\phi}:\partial X\rightarrow \partial Y$ is injective.
\end{proof}

From Propositions~\ref{Prop2-3} and \ref{Prop2-6}, 
we obtain that 
the map $\bar{\phi}:\partial X\rightarrow \partial Y$ is bijective.

Then we show the following.

\begin{Proposition}~\label{Prop2-7}
The map $\bar{\phi}:\partial X\rightarrow \partial Y$ is continuous.
\end{Proposition}

\begin{proof}
Let $\alpha\in \partial X$ and 
let $\{\alpha_i\,|\,i\in \N\}\subset \partial X$ be a sequence 
which converges to $\alpha$ in $\partial X$.

It is sufficient to show that the sequence $\{\bar{\alpha}_i\,|\,i\in \N\}$ 
converges to $\bar{\alpha}$ in $\partial Y$, where 
$\bar{\alpha}_i=\bar{\phi}(\alpha_i)$ and $\bar{\alpha}=\bar{\phi}(\alpha)$.

For each $i\in \N$, 
there exists a sequence $\{a_{i,j}x_0\,|\,j\in \N\}$ which converges to $\alpha_i$ as $j\rightarrow \infty$.

Since the sequence $\{\alpha_i\,|\,i\in \N\}\subset \partial X$ converges to $\alpha$ in $\partial X$, 
for any $r>0$, 
there exists $i_0\in \N$ such that 
$$\alpha_i \in U_{X\cup\partial X}(\alpha;r,1)$$ 
for any $i \ge i_0$.
Since the sequence $\{a_{i,j}x_0\,|\,j\in \N\}$ converges to $\alpha_i$ as $j\rightarrow \infty$ in $X\cup\partial X$, 
for any $i\ge i_0$, 
there exists $j_i\in \N$ such that 
$$ a_{i,j}x_0 \in U_{X\cup\partial X}(\alpha_i;r,1) $$ 
for any $j\ge j_i$.

Then 
$$ a_{i,j}x_0 \in U_{X\cup\partial X}(\alpha;r,2) $$ 
for any $i\ge i_0$ and $j\ge j_i$.

Now the sequence $\{a_{i,j}x_0\,|\, j\ge j_i, i\ge i_0\}$ converges to $\alpha$ 
as $i\rightarrow \infty$ in $X\cup\partial X$.
Hence the sequence $\{a_{i,j}y_0\,|\, j\ge j_i, i\ge i_0\}$ converges to $\bar{\phi}(\alpha)=\bar{\alpha}$ 
as $i\rightarrow \infty$ in $Y\cup\partial Y$.

Here we note that 
the sequence $\{a_{i,j}y_0\,|\,j\in \N\}$ converges to $\bar{\phi}(\alpha_i)=\bar{\alpha}_i$ 
as $j\rightarrow \infty$ in $Y\cup\partial Y$.

Then for any $r>0$, 
there exists an enough large number $i\ge i_0$ such that 
$$ a_{i,j}y_0 \in U_{Y\cup \partial Y}(\bar{\alpha};r,1) $$
for any $j \ge j_i$.
Also there exists an enough large number $j \ge j_i$ such that 
$$ a_{i,j}y_0 \in U_{Y\cup \partial Y}(\bar{\alpha}_i;r,1). $$
Then we obtain that 
$$ \bar{\alpha}_i \in U_{Y\cup \partial Y}(\bar{\alpha};r,2). $$

Thus the sequence $\{\bar{\alpha}_i\,|\,i\in \N\}$ 
converges to $\bar{\alpha}$ in $\partial Y$ 
and 
the map $\bar{\phi}:\partial X\rightarrow \partial Y$ is continuous.
\end{proof}

Therefore we obtain the following.

\begin{Theorem}
The map $\bar{\phi}:\partial X\rightarrow \partial Y$ is a $G$-equivariant homeomorphism.
\end{Theorem}

\begin{proof}
By the argument above, 
the map $\bar{\phi}:\partial X\rightarrow \partial Y$ is well-defined, bijective and continuous.

From the definition and the well-definedness of $\bar{\phi}$, 
we obtain that 
the map $\bar{\phi}:\partial X\rightarrow \partial Y$ is $G$-equivariant.
Indeed for any $\alpha\in\partial X$ and $g\in G$, 
if $\{g_ix_0\}\subset Gx_0\subset X$ is a sequence which converges to $\alpha$ in $X\cup\partial X$, 
then $\bar{\phi}(\alpha)$ is the point of $\partial Y$ 
to which the sequence $\{g_iy_0\}\subset Gy_0\subset Y$ converges in $Y\cup\partial Y$.
Then $\{gg_ix_0\}\subset Gx_0\subset X$ is the sequence which converges to $g\alpha$ in $X\cup\partial X$ 
and $\bar{\phi}(g\alpha)$ is the point of $\partial Y$ 
to which the sequence $\{gg_iy_0\}\subset Gy_0\subset Y$ converges in $Y\cup\partial Y$.
Here we note that 
the sequence $\{gg_iy_0\}\subset Gy_0\subset Y$ converges to $g\bar{\phi}(\alpha)$ in $Y\cup\partial Y$ 
by the definition of the action of $G$ on $\partial Y$.
Hence $\bar{\phi}(g\alpha)=g\bar{\phi}(\alpha)$ 
for any $\alpha\in\partial X$ and $g\in G$ and 
the map $\bar{\phi}:\partial X\rightarrow \partial Y$ is $G$-equivariant.

Also, 
the map $\bar{\phi}:\partial X\rightarrow \partial Y$ is closed, 
since $\partial X$ and $\partial Y$ are compact and metrizable.

Therefore, we obtain that 
the map $\bar{\phi}:\partial X\rightarrow \partial Y$ is 
a $G$-equivariant homeomorphism.
\end{proof}

\section{Proof of Theorem~\ref{Thm2}}

In this section, we prove Theorem~\ref{Thm2}.

We suppose that a group $G$ acts geometrically on two CAT(0) spaces $(X,d_X)$ and $(Y,d_Y)$.
Let $x_0\in X$ and $y_0\in Y$.

We first show that if the condition~$(**)$ holds, 
then there exists a $G$-equivariant homeomorphism of 
the boundaries $\partial X$ and $\partial Y$ 
as a continuous extension of the quasi-isometry 
$\phi:Gx_0\rightarrow Gy_0$ defined by $\phi(gx_0)=gy_0$.

Now we suppose that the condition $(**)$ holds; that is, 
\begin{enumerate}
\item[$(**)$] 
For any sequence $\{g_i\,|\,i\in \N\} \subset G$, 
the sequence $\{g_i x_0\,|\,i\in \N\}$ is a Cauchy sequence in $X\cup \partial X$ if and only if 
the sequence $\{g_i y_0\,|\,i\in \N\}$ is a Cauchy sequence in $Y\cup \partial Y$.
\end{enumerate}

We define a map 
$\bar{\phi}:\partial X\rightarrow \partial Y$ 
as a continuous extension of the quasi-isometry 
$\phi:Gx_0\rightarrow Gy_0$ defined by $\phi(gx_0)=gy_0$ by the same argument in Section~3:
For each $\alpha\in \partial X$, 
there exists a sequence $\{g_ix_0\}\subset Gx_0\subset X$ 
which converges to $\alpha$ in $X\cup\partial X$.
Then the sequence $\{g_ix_0\}\subset X$ is a Cauchy sequence in $X\cup \partial X$.
By the condition~$(**)$, 
the sequence $\{g_iy_0\}\subset Y$ is also a Cauchy sequence in $Y\cup\partial Y$.
Hence the sequence $\{g_iy_0\}\subset Y$ converges to 
some point $\bar{\alpha}\in\partial Y$ in $Y\cup\partial Y$.
Then we define $\bar{\phi}(\alpha)=\bar{\alpha}$.

We obtain the following by the same proof as the one of Proposition~\ref{Prop2-2}.

\begin{Proposition}\label{Prop4-2}
The map $\bar{\phi}:\partial X\rightarrow \partial Y$ is well-defined.
\end{Proposition}

Also we can define the inverse ${\bar{\phi}}^{-1}:\partial Y \rightarrow \partial X$ 
as a continuous extension of the quasi-isometry 
$\phi^{-1}:Gy_0\rightarrow Gx_0$ defined by $\phi^{-1}(gy_0)=gx_0$ by the condition~$(**)$.

Hence, we obtain the following.

\begin{Proposition}\label{Prop4-3}
The map $\bar{\phi}:\partial X\rightarrow \partial Y$ is bijective.
\end{Proposition}

Also, we obtain the following by the same proof as the one of Proposition~\ref{Prop2-7}.

\begin{Proposition}~\label{Prop4-4}
The map $\bar{\phi}:\partial X\rightarrow \partial Y$ is continuous.
\end{Proposition}

Thus we obtain the following theorem.

\begin{Theorem}
The map $\bar{\phi}:\partial X\rightarrow \partial Y$ is a $G$-equivariant homeomorphism.
\end{Theorem}

Conversely, 
we suppose that the condition~$(**)$ does not hold.

Then there exists a sequence $\{g_i\,|\,i\in \N\}\subset G$ 
such that either 
\begin{enumerate}
\item[(1)] $\{g_ix_0\,|\,i\in \N\}$ is a Cauchy sequence in $X\cup\partial X$ and 
$\{g_iy_0\,|\,i\in \N\}$ is {\it not} a Cauchy sequence in $Y\cup\partial Y$, or 
\item[(2)] $\{g_ix_0\,|\,i\in \N\}$ is {\it not} a Cauchy sequence in $X\cup\partial X$ and 
$\{g_iy_0\,|\,i\in \N\}$ is a Cauchy sequence in $Y\cup\partial Y$.
\end{enumerate}

Now we suppose that 
(1) $\{g_ix_0\,|\,i\in \N\}$ is a Cauchy sequence in $X\cup\partial X$ and 
$\{g_iy_0\,|\,i\in \N\}$ is {\it not} a Cauchy sequence in $Y\cup\partial Y$.

Then, 
the sequence $\{g_ix_0\,|\,i\in \N\}$ is Cauchy and converges to some point $\alpha \in \partial X$ 
in $X\cup\partial X$.
On the other hand, 
the sequence $\{g_iy_0\,|\,i\in \N\}$ is not Cauchy and 
contains subsequences $\{g_{i_j}y_0\,|\,j\in \N\}$ and $\{g_{k_j}y_0\,|\,j\in \N\}$ 
which converge to some points $\bar{\alpha}$ and $\bar{\beta}$ in $\partial Y$ respectively, 
with $\bar{\alpha} \neq \bar{\beta}$.

This means that the map $\phi:Gx_0\rightarrow Gy_0 \ (gx_0 \mapsto gy_0)$ does not continuously 
extend to any map $\bar{\phi}:\partial X \rightarrow \partial Y$.

By the same argument, 
we obtain that if 
(2) $\{g_ix_0\,|\,i\in \N\}$ is {\it not} a Cauchy sequence in $X\cup\partial X$ and 
$\{g_iy_0\,|\,i\in \N\}$ is a Cauchy sequence in $Y\cup\partial Y$, 
then the map $\phi^{-1}:Gy_0\rightarrow Gx_0 \ (gy_0 \mapsto gx_0)$ does not continuously 
extend to any map $\bar{\phi}^{-1}:\partial Y \rightarrow \partial X$.

Therefore, 
there does not exist a $G$-equivariant homeomorphism of 
the boundaries $\partial X$ and $\partial Y$ 
as a continuous extension of the quasi-isometry 
$\phi:Gx_0\rightarrow Gy_0$ defined by $\phi(gx_0)=gy_0$.

\section{Examples of equivariant rigid CAT(0) groups}

As an application of Theorem~\ref{Thm1} and the condition $(*)$, 
we introduce some examples.

\begin{Example}\label{Ex1}
Let $G=\Z\times \Z$.
The group $G$ acts geometrically on the flat plane $X=\R^2$ by 
$(a,b)\cdot (x,y)=(x+a,y+b)$ 
for any $(a,b)\in\Z\times \Z=G$ and $(x,y)\in \R^2=X$.

Suppose that the group $G$ acts geometrically on a CAT(0) space $Y$.
Then by the Flat Torus Theorem \cite[Theorem~II.7.1]{BH}, 
there exist a quasi-convex subset $Y'$ of $Y$ and a point $y_0\in Y$ such that 
$Y'$ is isometric to $\R^2$ and $Y'$ is the convex hull $C(Gy_0)$ of the orbit $Gy_0$ in $Y$, 
where the orbit $Gy_0$ is a lattice in $Y'\cong \R^2$.

Then there exists a linear transformation 
$\phi:X\rightarrow Y'$.

In this case, we see that 
the condition~$(*)$ holds for the actions of the group $G$ on $X$ and $Y$.

In fact, the induced map $\bar{\phi}:\partial X \rightarrow \partial Y$ 
of the boundaries $\partial X$ and $\partial Y$ that are homeomorphic to a circle $\S^1$ 
is an equivariant homeomorphism.
\end{Example}

\begin{Example}
By the same argument as above, 
we obtain that any geometric actions of a group $G=\Z^n$ ($n\in \N$) on any CAT(0) spaces 
satisfy the condition~$(*)$.
Hence $G=\Z^n$ is an equivariant rigid CAT(0) group.
\end{Example}

\begin{Example}
Let $G=(\Z\times \Z)*\Z_2$ which is the free product of $\Z\times \Z$ and $\Z_2$.
Let $G_1=\Z\times \Z$ and $A=\Z_2$; that is, $G=G_1*A$.

We construct a CAT(0) cubical cell complex $\Sigma$ as follows:
Let $X_1=\R^2$ on which $G_1=\Z\times \Z$ acts geometrically by 
$(a,b)\cdot (x,y)=(x+a,y+b)$ 
for each $(a,b)\in G_1$ and $(x,y)\in X_1$.
Here we consider that $X_1=\R^2$ is the cubical cell complex 
whose $1$-skeleton is the Cayley graph of $G_1=\Z\times \Z$.
Also let $X_2=[0,1]$ on which $G_2=\Z_2$ acts by 
$\bar{0} \cdot x=x$ and $\bar{1} \cdot x=1-x$ 
for each $\bar{0},\bar{1}\in \Z_2$ and $x\in [0,1]=X_2$; 
that is, $A=\Z_2$ is a reflection group of $X_2=[0,1]$.
Here we may consider that $X_2=[0,1]$ is the cubical complex 
whose $1$-skeleton is the Cayley graph of $A=\Z_2$.
Then we define the 2-dimensional cubical cell complex $\Sigma$ as 
$$ \Sigma = \bigcup\{gX_1\,|\, g\in G\} \cup \bigcup\{gX_2\,|\, g\in G\}, $$
where we identify the two points $g\cdot (0,0)\in gX_1$ and $g\cdot 0\in g X_2$ 
for any $g\in G$ and 
the $1$-skeleton of $\Sigma$ is the Cayley graph of $G=(\Z\times \Z)*\Z_2$.
This construction is similar to one of 
the Davis complex of the right-angled Coxeter group $W=((\Z_2*\Z_2)\times (\Z_2*\Z_2))*\Z_2$.

Now we show that 
if the group $G=(\Z\times \Z)*\Z_2$ acts geometrically on a CAT(0) space $Y$, 
then the actions of $G$ on $\Sigma$ and $Y$ satisfy the condition~$(*)$.

Suppose that the group $G=G_1*A=(\Z\times \Z)*\Z_2$ acts geometrically on a CAT(0) space $Y$.
Then there exists $y_0\in Y$ such that the convex-hull $C(G_1y_0)$ is isometric to $\R^2$ 
by \cite[Theorem~II.7.1]{BH}. We put $Y_1=C(G_1y_0)$.

Let identify $A=\{1,a\}$ ($a^2=1$).
Then we note that $\partial Y_1 \cap \partial aY_1 = \emptyset$.
Hence 
there exists $M>0$ such that 
\begin{align*}
&[y,y'] \cap B(y_0,M) \neq \emptyset \ \text{and} \\
&[y,y'] \cap B(ay_0,M) \neq \emptyset
\end{align*}
for any $y\in Y$ and $y'\in aY$.

Here for any $g\in G$, 
we can write $g=g_1a_1\cdots g_na_n$ for some $g_i\in G_1$ and $a_i\in A$ 
(where it may $g_1=1$ or $a_n=1$).
Then 
$$ [x_0,gx_0]=[x_0,g_1x_0]\cup[g_1x_0,g_1a_1x_0]\cup\dots \cup [g_1a_1\cdots g_nx_0,gx_0], $$
in $\Sigma$.
Also, 
\begin{align*}
&[y_0,gy_0] \cap B(g_1a_1\cdots g_{i-1}a_{i-1} g_i y_0,M) \neq \emptyset \ \text{and} \\
&[y_0,gy_0] \cap B(g_1a_1\cdots g_ia_i y_0,M) \neq \emptyset
\end{align*}
for any $i=1,\dots,n$ in $Y$.

Thus, we obtain that 
the geometric action of $G$ on $\Sigma$ and 
any geometric action of $G$ on any CAT(0) space $Y$ satisfy the condition~$(*)$.

Therefore, the group $G=(\Z\times \Z)*\Z_2$ is an equivariant rigid CAT(0) group.
\end{Example}

\begin{Example}
By the same argument as above, 
we obtain that groups of the form $G=\Z^n*\Z_2$ ($n\in \N$) 
are equivariant rigid CAT(0) groups.
\end{Example}

\begin{Example}
Let $G=\Z^n*A$ which is the free product of $\Z^n$ and a finite group $A$ and 
let $G_1=\Z^n$; that is, $G=G_1*A$.

We construct a CAT(0) cubical cell complex $\Sigma$ as follows:

Let $X_1=\R^n$ on which $G_1=\Z^n$ acts geometrically by 
$(a_1,\dots,a_n)\cdot (x_1,\dots,x_n)=(x_1+a_1,\dots,x_n+a_n)$ 
for any $(a_1,\dots,a_n)\in G_1$ and $(x_1,\dots,x_n)\in X_1$.
Here we consider that $X_1=\R^n$ is the cubical cell complex 
whose $1$-skeleton is the Cayley graph of $G_1=\Z^n$.
Then we consider the set $\{gX_1\,|\,g\in G\}$, 
where each $gX_1$ is a copy of $X_1$ 
and $gX_1=hX_1$ if and only if $g^{-1}h\in G_1$.
Let $x_0:=0$ in $X_1=\R^n$.
Then $G_1x_0$ is a lattice of $X_1=\R^n$.

Also let $X_2$ be the cone $x*Az_0$ of $Az_0=\{az_0\,|\,a\in A\}$ where 
the length of $[x,az_0]$ is $1$ and $[x,az_0]$ is isometric to $[0,1]$.
Here $A$ acts naturally by isometries on $X_2$ 
by $a\cdot x=x$ and $a\cdot bz_0=abz_0$ for any $a\in A$ and $bz_0\in Az_0\subset X_2$.
We may consider that $X_2$ is a $1$-dimensional cubical complex.
Then we consider the set $\{gX_2\,|\,g\in G\}$, 
where each $gX_2$ is a copy of $X_2$ 
and $gX_2=hX_2$ if and only if $g^{-1}h\in A$.

Here for each $g\in G$, 
we glue $gX_1$ and $gX_2$ by the one-point union 
$$gX_1 \vee_{gx_0=gz_0} gX_2.$$
Also we define the $n$-dimensional cubical cell complex $\Sigma$ as 
$$ \Sigma = \bigcup\{gX_1\,|\, g\in G\} \cup \bigcup\{gX_2\,|\, g\in G\}, $$
where we identify $gx_0=gz_0$ for any $g\in G$.
Then $\Sigma$ is contractible, since $G=G_1*A$ is the free-product of $G_1$ and $A$.
Moreover, $\Sigma$ is a CAT(0) space on which $G=G_1*A$ naturally acts geometrically.

Now we show that 
if the group $G=G_1*A$ acts geometrically on a CAT(0) space $Y$, 
then the actions of $G$ on $\Sigma$ and $Y$ satisfy the condition~$(*)$.

Suppose that the group $G=G_1*A$ acts geometrically on a CAT(0) space $Y$.
Then there exists $y_0\in Y$ such that the convex-hull $C(G_1y_0)$ is isometric to $\R^n$ 
by \cite[Theorem~II.7.1]{BH}.

Let $a\in A-\{1\}$.
Then we note that $\partial Y_1 \cap \partial aY_1 = \emptyset$.
Hence, by the structure of $G$, 
there exists $M>0$ such that 
\begin{align*}
&[y,y'] \cap B(y_0,M) \neq \emptyset \ \text{and} \\
&[y,y'] \cap B(ay_0,M) \neq \emptyset
\end{align*}
for any $y\in Y$ and $y'\in aY$.

For any $g\in G$, 
we can write $g=g_1a_1\cdots g_na_n$ for some $g_i\in G_1$ and $a_i\in A$.
Then 
$$ [x_0,gx_0]=[x_0,g_1x_0]\cup[g_1x_0,g_1a_1x_0]\cup\dots \cup [g_1a_1\cdots g_nx_0,gx_0], $$
in $\Sigma$.
Also, 
\begin{align*}
&[y_0,gy_0] \cap B(g_1a_1\cdots g_{i-1}a_{i-1} g_i y_0,M) \neq \emptyset \ \text{and} \\
&[y_0,gy_0] \cap B(g_1a_1\cdots g_ia_i y_0,M) \neq \emptyset
\end{align*}
for any $i=1,\dots,n$ in $Y$.

Hence, we obtain that 
the geometric action of $G$ on $\Sigma$ and 
any geometric action of $G$ on any CAT(0) space $Y$ satisfy the condition~$(*)$.

Therefore, the group $G=\Z^n*A$ is an equivariant rigid CAT(0) group.
\end{Example}

\begin{Example}
Groups of the form $G=\Z^{n_1}*\Z^{n_2}$ ($n_1,n_2\in \N$) 
are equivariant rigid CAT(0) groups.

First, for $G=\Z^{n_1}*\Z^{n_2}$ ($n_1,n_2\in \N$) 
where we put $G_1=\Z^{n_1}$ and $G_2=\Z^{n_2}$; that is $G=G_1*G_2$, 
we construct a CAT(0) cubical cell complex $\Sigma$ as follows:

For each $i=1,2$, 
let $X_i=\R^{n_i}$ on which $G_i=\Z^{n_i}$ acts geometrically by 
$(a_1,\dots,a_{n_i})\cdot (x_1,\dots,x_{n_i})=(x_1+a_1,\dots,x_{n_i}+a_{n_i})$ 
for any $(a_1,\dots,a_{n_i})\in G_i$ and $(x_1,\dots,x_{n_i})\in X_i$.
Here we consider that $X_i=\R^{n_i}$ is the cubical cell complex 
whose $1$-skeleton is the Cayley graph of $G_i=\Z^{n_i}$.
Then we consider the set $\{gX_i\,|\,g\in G\}$, 
where each $gX_i$ is a copy of $X_i$ 
and $gX_i=hX_i$ if and only if $g^{-1}h\in G_i$ for each $i=1,2$.
Let $x_i:=0$ in $X_i=\R^{n_i}$.
Then $G_ix_i$ is a lattice of $X_i=\R^{n_i}$.

Here for each $g\in G$, 
we glue $gX_1$ and $gX_2$ by the one-point union 
$$gX_1 \vee_{gx_1=gx_2} gX_2.$$
Also we define the $(\max\{n_1,n_2\})$-dimensional cubical cell complex $\Sigma$ as 
$$ \Sigma = \bigcup\{gX_1\,|\, g\in G\} \cup \bigcup\{gX_2\,|\, g\in G\}, $$
where we identify $gx_1=gx_2$ for any $g\in G$.
Then $\Sigma$ is contractible, since $G=G_1*G_2$ is the free-product of $G_1$ and $G_2$.
Moreover, $\Sigma$ is a CAT(0) space on which $G=G_1*G_2$ naturally acts geometrically.

Now we show that 
if the group $G=G_1*G_2$ acts geometrically on a CAT(0) space $Y$, 
then the actions of $G$ on $\Sigma$ and $Y$ satisfy the condition~$(*)$.

Suppose that the group $G=G_1*G_2$ acts geometrically on a CAT(0) space $Y$.
Then for each $i=1,2$, 
there exists $y_i\in Y$ such that the convex-hull $C(G_iy_i)$ is isometric to $\R^{n_i}$ 
by \cite[Theorem~II.7.1]{BH}.
Here we put $y_0:=y_1$, $Y_1:=C(G_1y_0)$ and $Y_2:=C(G_2y_0)$.
Then $Y_1$ is isometric to $\R^{n_1}$ on which $G_1$ acts geometrically by lattice.

We show that $G_2$ acts cocompactly on $Y_2$.
Let $M=d_Y(y_0,C(G_2y_2))$.
Then for any $g,g'\in G_2$, we have that 
$d_Y(gy_0,C(G_2y_2))\le M$ and $d_Y(g'y_0,C(G_2y_2))\le M$.
Hence $[gy_0,g'y_0]\subset B(C(G_2y_2),M)$, 
since $Y$ is a CAT(0) space and $C(G_2y_2)$ is convex.
Also for any $y,y'\in Y_2=C(G_2y_0)$, 
if $d_Y(y,C(G_2y_2))\le M$ and $d_Y(y',C(G_2y_2))\le M$ then 
$[y,y']\subset B(C(G_2y_2),M)$, since $Y$ is a CAT(0) space and $C(G_2y_2)$ is convex.
Thus we obtain that $Y_2=C(G_2y_0)\subset B(C(G_2y_2),M)$.
Let $N>0$ such that $G_2 B(y_2,N)=Y_2$.
Then we note that $B(C(G_2y_2),M)\subset G_2 B(y_2,M+N)$.
Hence 
$$Y_2=C(G_2y_0)\subset B(C(G_2y_2),M)\subset G_2 B(y_2,M+N).$$
Here $B(y_2,M+N)$ is compact.
Thus the action of $G_2$ on $Y_2$ is cocompact.
This implies that the group $G_2$ acts geometrically on the CAT(0) space $Y_2$.

By \cite[Theorem~II.7.1]{BH}.
there exists $y'_2\in Y_2$ such that the convex-hull $C(G_2y'_2)\subset Y_2$ is isometric to $\R^{n_2}$.
Let $Y'_2:=C(G_2y'_2)$.

Since the action of $G$ on $Y$ is cocompact, 
there exists $\bar{N}>0$ such that $G B(y_0,\bar{N})=Y$.
Then we put $\bar{Y_1}=B(Y_1,\bar{N})$ and $\bar{Y_2}=B(Y_2,\bar{N})$.
Here we note that $y_0\in Y_1\cap Y_2$ and $G (\bar{Y_1}\cap\bar{Y_2})=Y$.
Hence 
\begin{align*}
Y&=\bigcup\{g\bar{Y_1}\,|\,g\in G\} \ \text{and} \\
Y&=\bigcup\{g\bar{Y_2}\,|\,g\in G\}.
\end{align*}

Here we show that 
the set $\bar{Y_1}\cap \bar{Y_2}$ is bounded.
Indeed if $\bar{Y_1}\cap \bar{Y_2}$ is unbounded, then 
there exists $\alpha\in \partial (\bar{Y_1}\cap \bar{Y_2})$.
Then for $i=1,2$, there exists a sequence $\{g_jy_0\,|\,j\in\N\}\subset G_iy_0$ 
which converges to $\alpha$ in $Y\cup\partial Y$.
Hence $\alpha\in L(G_i)$ for each $i=1,2$, 
where $L(G_i)=\overline{G_iy_0} \cap \partial Y$ which is the limit set of $G_i$ in $\partial Y$.
Then by \cite[Theorem~4.15]{Ho0}, 
$$\alpha\in L(G_1)\cap L(G_2)=L(G_1\cap G_2)=L(\{1\})=\emptyset,$$
because $G=G_1*G_2$ is the free-product of $G_1$ and $G_2$.
This is a contradiction.
Thus we obtain that $\bar{Y_1}\cap \bar{Y_2}$ is bounded.

We also note that $gG_1g^{-1}\cap hG_2h^{-1}=\{1\}$ for any $g,h\in G$ and 
$gG_ig^{-1}\cap hG_ih^{-1}=\{1\}$ 
for any $g,h\in G$ such that $gG_i\neq hG_i$ ($i=1,2$).
Hence, by a similar argument to above one, 
we also obtain that $g\bar{Y_1}\cap h\bar{Y_2}$ is bounded for any $g,h\in G$, and 
$g\bar{Y_i}\cap h\bar{Y_i}$ is bounded 
for any $g,h\in G$ such that $gG_i\neq hG_i$ ($i=1,2$).

By \cite[Theorem~I.8.10]{BH}, 
the set $$S=\{s\in G\,|\, B(x_0,\bar{N})\cap B(sx_0,\bar{N})\neq\emptyset\}-\{1\}$$
is a generating set of $G$.
We can write $S=\{s_1,\dots,s_k\}\cup\{s'_1,\dots,s'_l\}$ where 
$S_1=\{s_1,\dots,s_k\}$ generates $G_1$ and 
$S_2=\{s'_1,\dots,s'_l\}$ generates $G_2$.
Here $S_1\cap S_2=\emptyset$, since $G=G_1*G_2$.

For any $g\in G$, 
we can write $g=g_1h_1\cdots g_nh_n$ for some $g_i\in G_1$ and $h_i\in G_2$ 
(where it may $g_1=1$ or $h_n=1$).
Then 
$$ [x_0,gx_0]=[x_0,g_1x_0]\cup[g_1x_0,g_1h_1x_0]\cup\dots \cup [g_1h_1\cdots g_nx_0,gx_0], $$
in $\Sigma$.
By the structure of $G$ and the argument above, we have that 
\begin{align*}
&[y_0,gy_0] \cap B(g_1h_1\cdots g_{i-1}h_{i-1} g_i y_0,\bar{N}) \neq \emptyset \ \text{and} \\
&[y_0,gy_0] \cap B(g_1h_1\cdots g_ih_i y_0,\bar{N}) \neq \emptyset
\end{align*}
for any $i=1,\dots,n$ in $Y$.

Hence, we obtain that 
the geometric action of $G$ on $\Sigma$ and 
any geometric action of $G$ on any CAT(0) space $Y$ satisfy the condition~$(*)$.

Therefore, the group $G=G_1*G_2$ is an equivariant rigid CAT(0) group.
\end{Example}

\begin{Example}
By the same argument as above, 
we obtain that groups of the form 
$$G=\Z^{n_1}*\dots *\Z^{n_k}$$ 
($n_i\in \N$) 
are equivariant rigid CAT(0) groups.

Moreover, we can obtain that 
groups of the form 
$$G=\Z^{n_1}*\dots *\Z^{n_k}*A_1*\dots * A_l$$ 
(where $n_i\in \N$ and each $A_j$ is a finite group) 
are equivariant rigid CAT(0) groups.
\end{Example}

\section{Examples of non equivariant rigid CAT(0) groups}

As an application of Theorem~\ref{Thm2} and the condition $(**)$, 
we introduce some examples.

\begin{Example}\label{example1}
Let $G= F_2\times \Z$, where $F_2$ is the rank 2 free group generated by $\{a,b\}$.
Let $T$ and $T'$ be the Cayley graphs of $F_2$ with respect to the generating set $\{a,b\}$ 
such that 
\begin{enumerate}
\item in $T$, all edges $[g,ga]$ and $[g,gb]$ ($g\in F_2$) have the unit length, and 
\item in $T'$, the length of $[g,ga]$ is $2$ and the length of $[g,gb]$ is $1$ for any $g\in F_2$
\end{enumerate}
(see Figure~1). Here we note that $F_2$ acts naturally and geometrically on $T$ and $T'$.

\begin{figure}[htbp]
\begin{center}
\unitlength=0.9mm
\begin{picture}(130,60)
\put(0,30){\line(1,0){40}}
\put(20,10){\line(0,1){40}}
\put(14,38){\line(1,0){12}}
\put(14,22){\line(1,0){12}}
\put(12,24){\line(0,1){12}}
\put(28,24){\line(0,1){12}}
\put(18,43){\line(1,0){4}}
\put(18,17){\line(1,0){4}}
\put(7,28){\line(0,1){4}}
\put(33,28){\line(0,1){4}}
\put(9,27){\line(1,0){6}}
\put(9,33){\line(1,0){6}}
\put(25,27){\line(1,0){6}}
\put(25,33){\line(1,0){6}}
\put(17,19){\line(0,1){6}}
\put(23,19){\line(0,1){6}}
\put(17,35){\line(0,1){6}}
\put(23,35){\line(0,1){6}}
\put(18,3){$T$}
\put(50,30){\line(1,0){80}}
\put(90,10){\line(0,1){40}}
\put(78,38){\line(1,0){24}}
\put(78,22){\line(1,0){24}}
\put(74,24){\line(0,1){12}}
\put(106,24){\line(0,1){12}}
\put(86,43){\line(1,0){8}}
\put(86,17){\line(1,0){8}}
\put(64,28){\line(0,1){4}}
\put(116,28){\line(0,1){4}}
\put(68,27){\line(1,0){12}}
\put(68,33){\line(1,0){12}}
\put(100,27){\line(1,0){12}}
\put(100,33){\line(1,0){12}}
\put(84,19){\line(0,1){6}}
\put(96,19){\line(0,1){6}}
\put(84,35){\line(0,1){6}}
\put(96,35){\line(0,1){6}}
\put(88,3){$T'$}
\end{picture}
\end{center}
\caption{}
\label{fig1}
\end{figure}

Let $X=T\times \R$ and $Y=T'\times \R$.
We consider the natural actions of the group $G$ on the CAT(0) spaces $X$ and $Y$.

Then the group $G$ acts geometrically on the two CAT(0) spaces $X$ and $Y$, 
and the quasi-isometry $gx_0 \mapsto gy_0$ 
(where $x_0=(1,0)\in X$ and $y_0=(1,0)\in Y$) 
does not extend continuously to 
any map from $\partial X$ to $\partial Y$.

Indeed, for example, 
we can consider the sequence $\{g_n\,|\, n \in \N \} \subset F_2$ such that 
$g_1=ab$ and
$$ g_n= 
\begin{cases}
g_{n-1}a^{2^{n-1}} & \mbox{if $n$ is even}\\
g_{n-1}b^{2^{n-1}} & \mbox{if $n$ is odd}
\end{cases}
$$
for $n\ge 2$.
Here we note that the length of the words of $g_n$ in $F_2$ is $2^n$.
Let $\bar{g}_n=(g_n,2^n) \in F_2\times \Z$ for $n\in \N$.
Then the sequence $\{\bar{g}_n x_0\}$ is a Cauchy sequence in $X\cup\partial X$ and 
converges to the point $[\alpha,\frac{\pi}{4}]$ 
where $\alpha \in \partial T$ to which $\{g_n\}$ converges.
On the other hand, 
the sequence $\{\bar{g}_n y_0\}$ is {\it not} a Cauchy sequence in $Y\cup\partial Y$.
Thus by the condition $(**)$ and Theorem~\ref{Thm2}, 
the map $\phi:Gx_0\rightarrow Gy_0 \ (gx_0 \mapsto gy_0)$ does not continuously 
extend to any map $\bar{\phi}:\partial X \rightarrow \partial Y$.

Here we note that the group 
$G= F_2\times \Z$ is a rigid CAT(0) group whose boundary is 
the suspension of the Cantor set.
\end{Example}

\begin{Example}\label{example2}
By the argument in Example~\ref{example1}, 
we obtain that every CAT(0) group of the form 
$G=F \times H$ 
where $F$ is a free group of rank $n\ge 2$ and $H$ is an infinite CAT(0) group, 
is non equivariant rigid.

Indeed 
we can consider the Cayley graphs $T$ and $T'$ of $F$ 
with respect to the generating set $\{a_1,\dots,a_n\}$ of $F$ 
such that 
\begin{enumerate}
\item in $T$, all edges $[g,ga_i]$ ($g\in F$, $i=1,\dots,n$) have the unit length, and 
\item in $T'$, the length of $[g,ga_1]$ ($g\in F$) is $2$ 
and the length of $[g,ga_i]$ ($g\in F$, $i=2,\dots,n$) is $1$.
\end{enumerate}
Let $Y$ be a CAT(0) space on which $H$ acts geometrically and 
let $X=T\times Y$ and $X'=T'\times Y$.
Since $H$ is an infinite CAT(0) group, 
there exists $h_0 \in H$ such that the order $o(h_0)=\infty$ (\cite[Theorem~11]{Sw}).
Then $h_0$ is a hyperbolic isometry of $Y$.
Let $\sigma_0$ be an axis for $h_0$ in $Y$.
Then, the natural actions of $G=F \times H$ on $X$ and $X'$ 
does not continuously 
extend to any map between boundaries $\partial X$ and $\partial X'$ 
by the same argument as Example~\ref{example1}.
Indeed 
$F \times \langle h_0 \rangle \subset F \times H =G$, 
$T \times \Image \sigma_0 \subset X$ and 
$T' \times \Image \sigma_0 \subset X'$.

Here we note that 
$\partial X=\partial T * \partial Y$ and 
$\partial X'=\partial T' * \partial Y$ are homeomorphic.
Also, if $G=F\times H$ acts geomretrically on a CAT(0) space $X$ 
then the boundary $\partial X$ is homeomorphic to $\partial F * \partial Z$ 
where $Z$ is some CAT(0) space on which $H$ acts geometrically 
by the splitting theorem in \cite{Ho}.
Hence $G$ is rigid if and only if $H$ is rigid.
\end{Example}

\section{Remarks and questions}

The author thinks that 
the main theorems, the conditions $(*)$ and $(**)$ 
and some arguments in this paper 
could be used to investigate boundaries of CAT(0) groups and interesting open problems on 
\begin{enumerate}
\item[(1)] (equivariant) rigidity of boundaries of CAT(0) groups; 
\item[(2)] (equivariant) rigidity of boundaries of Coxeter groups; 
\item[(3)] (equivariant) rigidity of boundaries of Davis complexes of Coxeter groups; 
\item[(4)] (equivariant) rigidity of boundaries of CAT(0) spaces 
on which Coxeter groups act geometrically as reflection groups;
\item[(5)] (equivariant) rigidity of boundaries of CAT(0) spaces 
on which right-angled Coxeter groups act geometrically as reflection groups;
\item[(6)] (equivariant) rigidity of boundaries of CAT(0) cubical complexes 
on which CAT(0) groups act geometrically;
\item[(7)] (equivariant) rigidity of boundaries of CAT(0) {\it cuboidal} complexes 
on which CAT(0) groups act geometrically, 
\end{enumerate}
etc.

Here we can find some recent research on CAT(0) groups and their boundaries 
in \cite{CS}, \cite{CK}, \cite{GO}, \cite{Ha}, \cite{Ho}, \cite{Hr}, \cite{MR}, 
\cite{Mon}, \cite{Moo1}, \cite{Moo2}, \cite{PS} and \cite{W}.
Details of Coxeter groups and Coxeter systems are 
found in \cite{Bo}, \cite{Br} and \cite{Hu}, and 
details of Davis complexes which are CAT(0) spaces defined by Coxeter systems 
and their boundaries are 
found in \cite{D1}, \cite{D2} and \cite{M}.
We can find some recent research on boundaries of Coxeter groups 
in \cite{CF}, \cite{Dr}, \cite{Dr2}, \cite{F}, \cite{Hos}, \cite{MRT}.
Here we note that 
every cocompact discrete reflection group of a geodesic space 
becomes a Coxeter group \cite{Ho5}, and 
if a Coxeter group $W$ is a cocompact discrete reflection group 
of a CAT(0) space $X$, 
then the CAT(0) space $X$ has a structure similar to the Davis complex \cite{Ho6}.

\vspace*{3mm}

The following theorem is known.

\begin{Theorem}[{\cite[Theorem 5.2]{Ho}}]
If $G_1$ and $G_2$ are boundary rigid CAT(0) groups, 
then so is $G_1\times G_2$.
\end{Theorem}

On research of (equivariant) boundary rigidity of CAT(0) groups, 
the following natural open problem is important.

\begin{Problem}
If $G_1$ and $G_2$ are (equivariant) boundary rigid CAT(0) groups, then is $G_1*G_2$ also?
\end{Problem}

We consider this problem for (4) above in Section~8 
and we provide a conjecture on this problem for (7) above in Section~9.

\section{On equivariant rigid as reflection groups}

A Coxeter group $W$ is said to be {\it equivariant rigid as a reflection group}, if 
for any two CAT(0) spaces $X$ and $Y$ on which $W$ acts geometrically 
as a {\it cocompact discrete reflection group} of $X$ and $Y$ (cf.\ \cite{Ho5}), 
it obtain a $W$-equivariant homeomorphism of 
the boundaries $\partial X$ and $\partial Y$ 
as a continuous extension of the quasi-isometry 
$\phi:Wx_0\rightarrow Wy_0$ defined by $\phi(wx_0)=wy_0$, 
where $x_0\in X$ and $y_0\in Y$.

Then we show the following.

\begin{Theorem}\label{Thm4}
The following statements hold.
\begin{enumerate}
\item[${\rm (i)}$] If Coxeter groups $W_1$ and $W_2$ are equivariant rigid 
as reflection groups, then $W_1*W_2$ is also.
\item[${\rm (ii)}$] For a Coxeter group $W=W_A*_{W_{A\cap B}}W_B$ where $W_{A\cap B}$ is finite, 
if $W$ determines its Coxeter system up to isomorphism, and 
if the parabolic subgroups $W_A$ and $W_B$ are equivariant rigid as reflection groups then $W$ is also.
\end{enumerate}
\end{Theorem}

\begin{proof}
(i) 
Let $W_1$ and $W_2$ be Coxeter groups that are equivariant rigid as reflection groups and let $W=W_1 * W_2$.
Suppose that the Coxeter group $W$ acts geometrically on two CAT(0) spaces $X$ and $Y$ 
as cocompact discrete reflection groups.

Let $(W,S)$ and $(W,S')$ be Coxeter systems obtained from 
the actions of $W$ on $X$ and $Y$ as \cite{Ho5} respectively, 
and let $C$ and $D$ be chambers as $W \overline{C} =X$ and $W \overline{D}=Y$.

Then $S$ and $S'$ separate as the disjoint unions $S=S_1 \cup S_2$ and $S'=S'_1 \cup S'_2$ 
such that $W_{S_1}=W_{S'_1}=W_1$ and $W_{S_2}=W_{S'_2}=W_2$, since $W=W_1*W_2$.

Let $X_i=W_i\overline{C}$ and $Y_i=W_i\overline{D}$ for $i=1,2$.
Then each Coxeter group $W_i$ is cocompact discrete reflection group of $X_i$ and $Y_i$ (cf.\ \cite{Ho6}).
Since $W_i$ is equivariant rigid as reflection groups, 
it obtain $W_i$-equivariant homeomorphism 
$\bar{\phi}_i:\partial X_i\rightarrow \partial Y_i$ 
as a continuous extension of the quasi-isometry 
$\phi_i:W_ix_0\rightarrow W_iy_0$ defined by $\phi_i(wx_0)=wy_0$, 
where $x_0\in C \subset X_1\cap X_2$ and $y_0\in D \subset Y_1\cap Y_2$.

Let $N=\diam \overline{C}$ and $M=\diam \overline{D}$. 
Then $WB(x_0,N)=X$ and $WB(y_0,M)=Y$.

Now we define a map $\bar{\phi}:\partial X \rightarrow \partial Y$ naturally as follows.

Let $\alpha \in \partial X$ and 
let $\xi_\alpha$ be the geodesic ray in $X$ 
with $\xi_\alpha(0)=x_0$ and $\xi_\alpha(\infty)=\alpha$.

By \cite[Theorem~3]{Ho6}, 
there exists a sequence $\{s_i\,|\,i\in\N\}\subset S$ 
such that 
each $w_n=s_1\cdots s_n$ is a reduced representation and 
$d_H(\Image\xi_\alpha,P) \le N$, 
where $d_H$ is the Hausdorff distance and 
$$P=[x_0,s_1x_0]\cup[s_1x_0,s_1s_2x_0]\cup\dots\cup[s_1\cdots s_{n-1}x_0,w_nx_0]\cup\cdots. $$

Here since $w_n\in W=W_1*W_2$ for any $n\in \N$, 
we can write $w_n=a_1b_1\cdots a_kb_k$ for some $a_i\in W_1$ and $b_i\in W_2$.

Then either 
\begin{enumerate}
\item[(1)] there exists $n_0\in \N$ such that $w_n \in w_{n_0}W_1$ for any $n\ge n_0$, 
that is, $\alpha \in w_{n_0} \partial X_1$, 
\item[(2)] there exists $n_0\in \N$ such that $w_n \in w_{n_0}W_2$ for any $n\ge n_0$, 
that is, $\alpha \in w_{n_0} \partial X_2$, or
\item[(3)] the sequences $\{a_i\}\subset W_1$ and $\{b_i\}\subset W_2$ are infinite, 
that is, $\alpha \not\in \bigcup\{w(\partial X_1\cup \partial X_2)\,|\,w\in W\}$.
\end{enumerate}

In the case (1), 
the sequence $\{w_{n_0}s_{n_0+1}\cdots s_m y_0\,|\,m> n_0\}$ converges to 
some point $\bar{\alpha}\in w_{n_0}\partial Y_1$, 
since $W_1$ is equivariant rigid as reflection groups.
Then we define $\bar{\phi}(\alpha)=\bar{\alpha}$.
(Here we note that $\bar{\phi}(\alpha)=w_{n_0}\bar{\phi}_1(\beta)$ 
where $\beta$ is the point of $\partial X_1$ 
to which the sequence $\{s_{n_0+1}\cdots s_m x_0\,|\,m> n_0\}$ 
converges in $X_1 \cup\partial X_1$.)

Also in the case (2), similarly, 
we define $\bar{\phi}(\alpha)=\bar{\alpha}$ 
where $\bar{\alpha}\in w_{n_0}\partial Y_2$ is the point 
to which the sequence $\{w_{n_0}s_{n_0+1}\cdots s_m y_0\,|\,m> n_0\}$ converges.

In the case (3), 
$$ s_1s_2s_3\cdots\cdots s_n\cdots\cdots = a_1b_1a_2b_2 \cdots a_kb_k \cdots \cdots, $$
where $a_i\in W_1$ and $b_i\in W_2$.
Here we note that $X_1\cap X_2=\overline{C}$ and 
$Y_1\cap Y_2=\overline{D}$ are bounded and compact.
By \cite[Theorem~3]{Ho6},
$$P \cap B(a_1b_1\cdots a_kb_k x_0,N) \neq\emptyset$$ 
for any $k\in\N$.
Then the sequence $\{a_1b_1\cdots a_kb_k y_0\,|\,k\in\N\}$ 
converges to some point $\bar{\alpha}\in \partial Y$ in $Y \cup\partial Y$, 
because for each $1<i<k$, 
$$ d(a_1b_1\cdots a_ib_i y_0,[y_0,a_1b_1\cdots a_kb_k y_0]) \le M $$
by \cite[Theorem~3]{Ho6}.
We define $\bar{\phi}(\alpha)=\bar{\alpha}$.

The map $\bar{\phi}:\partial X \rightarrow \partial Y$ defined above 
is well-defined, and by similar arguments in sections above, 
we can show that $\bar{\phi}$ is bijective, continuous and 
a $W$-equivariant homeomorphism.

Therefore the Coxeter group $W$ is equivariant rigid as a reflection group.

\vspace{2.5mm}

(ii) 
Let $W=W_A*_{W_{A\cap B}}W_B$ be a Coxeter group where $W_{A\cap B}$ is finite.
Suppose that $W$ determines its Coxeter system up to isomorphism, 
the parabolic subgroups $W_A$ and $W_B$ are equivariant rigid as reflection groups, 
and the Coxeter group $W$ acts geometrically on two CAT(0) spaces $X$ and $Y$ 
as cocompact discrete reflection groups.

Then Coxeter systems $(W,S)$ and $(W,S')$ are obtained from 
the actions of $W$ on $X$ and $Y$ as \cite{Ho5} respectively, 
and let $C$ and $D$ be chambers as $W \overline{C} =X$ and $W \overline{D}=Y$.
Since $W$ determines its Coxeter system up to isomorphism, 
two Coxeter systems $(W,S)$ and $(W,S')$ are isomorphic.

By \cite{Ho6}, $X_A:=W_A\overline{C}$ and $X_B:=W_B\overline{C}$ are convex subspaces of $X$,
$Y_A:=W_A\overline{D}$ and $Y_B:=W_B\overline{D}$ are convex subspaces of $Y$, and 
$W_A$ (resp.\ $W_B$) acts geometrically on the two CAT(0) spaces $X_A$ and $Y_A$ (resp.\ $X_B$ and $Y_B$) 
as cocompact discrete reflection groups.

Since $W_A$ and $W_B$ are equivariant rigid as reflection groups,
it obtain a $W_A$-equivariant homeomorphism 
$\bar{\phi}_A:\partial X_A\rightarrow \partial Y_A$ and 
a $W_B$-equivariant homeomorphism 
$\bar{\phi}_B:\partial X_B\rightarrow \partial Y_B$ 
as continuous extensions of the quasi-isometries 
$\phi_A:W_Ax_0\rightarrow W_Ay_0$ defined by $\phi_A(wx_0)=wy_0$ and 
$\phi_B:W_Bx_0\rightarrow W_By_0$ defined by $\phi_B(wx_0)=wy_0$ respectively, 
where $x_0\in C\subset X_A\cap X_B$ and $y_0\in D\subset Y_A\cap Y_B$ 
respectively.

Then 
we can define a $W$-equivariant homeomorphism 
$\bar{\phi}:\partial X \rightarrow \partial Y$ naturally as a similar construction to (i), 
since $W_{A\cap B}$ is finite and 
$X_A\cap X_B=W_{A\cap B}\overline{C}$ and 
$Y_A\cap Y_B=W_{A\cap B}\overline{D}$ are bounded and compact.
\end{proof}

As an application of Theorem~\ref{Thm4}, 
we introduce some examples.

\begin{Example}
By Theorem~\ref{Thm4}~(i), any group of the form 
$$W=W_1 * \dots * W_n$$ 
where each $W_i$ is a Gromov hyperbolic Coxeter group,
an affine Coxeter group or a finite Coxeter group, 
is equivariant rigid as a reflection group.
\end{Example}

\begin{Example}\label{exampleX}
By Theorem~\ref{Thm4}~(ii), any Coxeter group of the form 
$$W=(\cdots(W_{A_1} *_{W_{B_1}} W_{A_2})*_{W_{B_2}} W_{A_3})*_{W_{B_3}} \cdots)*_{W_{B_{n-1}}} W_{A_n}$$ 
where each $W_{A_i}$ is 
a Gromov hyperbolic Coxeter group, an affine Coxeter group or a finite Coxeter group, 
each $W_{B_i}$ is finite and $W$ determines its Coxeter system up to isomorphism, 
is equivariant rigid as a reflection group.
\end{Example}

\begin{Example}
For example, by Theorem~\ref{Thm4}~(ii) and Example~\ref{exampleX}, 
the Coxeter groups defined by the diagrams in Figure~2 are 
equivariant rigid as reflection groups.
Here these Coxeter groups determine their Coxeter systems up to isomorphism 
by \cite{Ho00} and \cite{Ra}.
\begin{figure}[htbp]
\begin{center}
\unitlength=0.65mm
{\scriptsize
\begin{picture}(120,45)(-5,-5)
\put(0,0){\line(1,0){50}}
\put(0,25){\line(1,0){50}}
\put(0,0){\line(0,1){25}}
\put(25,0){\line(0,1){25}}
\put(50,0){\line(0,1){25}}
%
%
%
{\small
\put(-1.5,-1){$\bullet$}
\put(23.5,-1){$\bullet$}
\put(48.5,-1){$\bullet$}
\put(-1.5,24){$\bullet$}
\put(23.5,24){$\bullet$}
\put(48.5,24){$\bullet$}
}
\put(11,-5){$2$}
\put(36,-5){$2$}
\put(11,27){$2$}
\put(36,27){$2$}
\put(-3,12){$2$}
\put(26,12){$2$}
\put(51,12){$2$}
\put(80,0){\line(1,0){30}}
\put(80,20){\line(1,0){30}}
\put(80,0){\line(0,1){20}}
\put(110,0){\line(0,1){20}}
\put(80,20){\line(3,-2){30}}
\put(80,20){\line(1,1){15}}
\put(110,20){\line(-1,1){15}}
{\small
\put(78.5,-1){$\bullet$}
\put(78.5,19){$\bullet$}
\put(108.5,-1){$\bullet$}
\put(108.5,19){$\bullet$}
\put(93.5,34){$\bullet$}
}
%
%
%
\put(77,9.5){$4$}
\put(93,-5){$4$}
\put(88.5,8.5){$4$}
\put(111,9.5){$2$}
\put(93.5,21){$4$}
\put(84,28){$4$}
\put(102.7,28){$2$}
\end{picture}
}
\end{center}
\caption{}
\label{fig2}
\end{figure}
\end{Example}

By similar arguments to Examples~\ref{example1} and \ref{example2}, 
we can construct non equivariant rigid Coxeter groups 
(as reflection groups).

\begin{Example}\label{exampleA}
Let $F=\Z_2*\Z_2*\Z_2$, let $G=\Z_2*\Z_2$ and let $W=F\times G$.
Then $W$ is a non equivariant rigid Coxeter group 
(as a reflection group).

Indeed we can consider the Cayley graphs $T$ and $T'$ of $F$ 
with respect to the generating set $\{a_1,a_2,a_3\}$ of $F$ 
such that 
\begin{enumerate}
\item in $T$, all edges $[w,wa_i]$ ($w\in F$, $i=1,2,3$) have the unit length,
\item in $T'$, the length of $[w,wa_1]$ ($w\in F$) is $2$ 
and the length of $[w,wa_i]$ ($w\in F$, $i=2,3$) is $1$.
\end{enumerate}
Let $X=T\times \R$ and $X'=T'\times \R$.
Then, the natural actions of $W=F \times G$ on $X$ and $X'$ 
does not continuously 
extend to any map between the boundaries $\partial X$ and $\partial X'$ 
by the same argument as Example~\ref{example1}.

Here we note that the group 
$W=F\times G$ is a rigid CAT(0) group whose boundary is 
the suspension of the Cantor set.
\end{Example}

\begin{Example}\label{exampleB}
By the arguments in Examples~\ref{example1}, \ref{example2} and \ref{exampleA},
every Coxeter group of the form $W=F\times G$ 
where $F=\Z_2*\dots*\Z_2$ with rank $n\ge 3$ and $G$ is an infinite Coxeter group, 
is a non equivariant rigid (as a reflection group).
\end{Example}

Here we note that if $W=F\times G$ acts geometrically on a CAT(0) space $X$ 
then the boundary $\partial X$ is homeomorphic to $\partial F * \partial Y$ 
where $Y$ is some CAT(0) space on which $G$ acts geometrically 
by the splitting theorem in \cite{Ho}.
Hence $W$ is rigid if and only if $G$ is rigid.


\section{Conjecture}

Now we consider the CAT(0) group $$G=(F_2\times \Z)*\Z_2$$ 
where $F_2$ is the free group of rank 2.

We note that $F_2\times \Z$ is a rigid CAT(0) group and 
non equivariant rigid.

Then the following conjecture arises.

\begin{Conjecture}
The group $G=(F_2\times \Z)*\Z_2$ 
will be a non-rigid CAT(0) group with uncountably many boundaries.
\end{Conjecture}

This conjecture comes from the following idea.

For $p \ge q \ge 1$, 
let $T_{p,q}$ be the Cayley graph of the free group $F_2$ with the generating set $\{a,b\}$ such that
\begin{enumerate}
\item[$\bullet$] 
the length of $[g,ga]$ is $p$ and the length of $[g,gb]$ is $q$ for any $g\in F$.
\end{enumerate}
Then $F_2\times \Z$ acts naturally on $T_{p,q}\times \R$.
By similar arguments to Examples~5.3--5.7,
we can construct a {\it cuboidal} cell complex $\Sigma_{p,q}$ on which 
$G=(F_2\times \Z)*\Z_2$ acts geometrically, 
where the 1-skeleton of $\Sigma_{p,q}$ is the Cayley graph of $G$ 
and $T_{p,q} \subset \Sigma_{p,q}^{(1)}$.

Then from the argument of Example~\ref{example1}, 
the author expects that 
if $\frac{p}{q} \neq \frac{p'}{q'}$ then 
the boundaries $\partial \Sigma_{p,q}$ and $\partial \Sigma_{p',q'}$ 
will be not homeomorphic.

\vspace*{4mm}

If this conjecture and this idea will be the case, 
then 
the right-angled Artin group 
$$(F_2 \times \Z)*\Z,$$ 
the right-angled Coxeter group 
$$(W_A \times W_B)*W_C$$ 
where $W_A=\Z_2*\Z_2*\Z_2$, $W_B=\Z_2*\Z_2$ and $W_C=\Z_2$, etc, 
will be also non-rigid CAT(0) groups with uncountably many boundaries.

\section{On rigidity}

Let $G$ and $H$ be groups acting 
geometrically (i.e.\ properly and cocompactly by isometries) 
on metric spaces $(X,d_X)$ and $(Y,d_Y)$ respectively.
We consider orbits $G x_0 \subset X$ and $Hy_0 \subset Y$ 
where $x_0\in X$ and $y_0\in Y$.

Let $\phi:G\to H$ be a map and 
let $\phi':Gx_0\to Hy_0$ $(gx_0\mapsto \phi(g)y_0)$.

Here if 
$X$ and $Y$ are Gromov hyperbolic spaces, CAT(0) spaces or Busemann spaces, 
then we can define the boundaries $\partial X$ and $\partial Y$.

Then it is well-known that 
if $\phi:G\to H$ is an isomorphism then 
$\phi':Gx_0\to Hy_0$ is a quasi-isometry  
and moreover if $G$ is Gromov hyperbolic then 
$\phi'$ induces an equivariant homeomorphism 
$\bar{\phi}: \partial X \to \partial Y$.

Theorem~\ref{Thm2} implies that 
if $\phi:G\to H$ is an isomorphism and 
the map $\phi':Gx_0\to Hy_0$ satisfies the condition~$(**)$
then $\phi'$ induces an equivariant homeomorphism 
$\bar{\phi}: \partial X \to \partial Y$.

$$
\begin{array}{ccccccccc}
G & \stackrel{\cdot}{\curvearrowright} & X & \supset &  Gx_0 & \longleftrightarrow & \partial X \\[1mm]
\ \ \ \downarrow{\phi}   & & & & \ \ \ \downarrow{\phi'}  & & \ \ \ \downarrow{\bar{\phi}}  \\[1mm]
H & \stackrel{\cdot}{\curvearrowright} & Y & \supset & Hy_0 & \longleftrightarrow & \partial Y 
\end{array}
$$

\vspace{4mm}

Then there are problems of rigidity.
\begin{enumerate}
\item[(I)] 
If $\phi:G\to H$ is an isomorphism then 
when does there exist an homeomorphism $\bar{\phi}: \partial X \to \partial Y$?
\item[(II)] 
If $\phi:G\to H$ is an isomorphism then 
when does $\phi'$ induce an equivariant homeomorphism 
$\bar{\phi}: \partial X \to \partial Y$?
\item[(III)] 
If $X=Y$ and $Gx_0=Hx_0$ then 
when are groups $G$ and $H$ virtually isomorphic 
(i.e.\ there exist finite-index subgroups $G'$ and $H'$ of $G$ and $H$ respectively 
such that $G'$ and $H'$ are isomorphic)?
\item[(IV)] 
If $X=Y$ and $Gx_0=Hx_0$ then 
when do there exist finite-index subgroups $G'$ and $H'$ of $G$ and $H$ respectively 
such that $G'$ and $H'$ are conjugate in the isometry group $\Isom(X)$ of $X$?
\item[(V)] 
If there is an isomorphism $\phi:G\to H$ then 
when does there exist a homeomorphism (or homotopy equivalence 
or strongly deformation retract) 
$\psi: X/G \to Y/H$?
\end{enumerate}

Here it seems that (III)--(V) are relate 
to \cite{BL}, \cite{Fu}, \cite{Fu2}, \cite{Gr-P}, 
\cite{K}, \cite{Mon}, \cite{M-S}, \cite{Mos} and \cite{Mos2}.

%

%
\end{document}